\documentclass[12pt]{amsart}
\usepackage{color}
\usepackage{dsfont}
\usepackage{url}
\usepackage{hyperref}
\usepackage{mathtools}
\mathtoolsset{showonlyrefs}
\usepackage{anysize}
\usepackage{geometry}
\marginsize{2cm}{2cm}{2cm}{2cm}
\usepackage[normalem]{ulem}
\theoremstyle{plain}
\numberwithin{equation}{section}
\newtheorem{thm}{Theorem}[section]
\newtheorem{cor}[thm]{Corollary}
\newtheorem{prop}[thm]{Proposition}
\newtheorem{lem}[thm]{Lemma}
\newtheorem{rem}[thm]{Remark}
\newtheorem{hyp}[thm]{Hypothesis}
\newcommand{\PP}[1]{\left(#1\right)}

\newcommand{\LL}[1]{\left\{#1\right\}}
\newcommand{\E}[1]{\mathbb{E}\left(#1\right)}
\newcommand{\abs}[1]{\left|#1\right|}
\renewcommand{\Pr}[1]{\mathbb{P}\left(#1\right)}

\newcommand{\Ind}[1]{\mathds{1}_{\left\{#1\right\}}}
\newcommand{\R}{\mathbb{R}}

\newcommand{\N}{\mathbb{N}}
\newcommand{\fa}{\mathfrak{a}}
\newcommand{\fs}{\mathfrak{s}}
\newcommand{\fb}{\mathfrak{b}}

\newcommand{\cA}{\mathcal{A}}
\newcommand{\cF}{\mathcal{F}}
\newcommand{\cD}{\mathcal{D}}
\newcommand{\NN}{\mathcal{N}}

\newcommand{\la}{\lambda}
\newcommand{\om}{\omega} 
\title[Limiting distribution of the condition number for random circulant matrices]{The asymptotic distribution of the condition number for random circulant matrices}
\author{Gerardo Barrera}
\address{University of Helsinki, Department of Mathematics and Statistics.
PL 68, Pietari Kalmin katu 5, 00014. Helsinki, Finland}
\email{gerardo.barreravargas@helsinki.fi}
\author{Paulo Manrique-Mir\'on}
\address{
National Polytechnic Institute.
Postal code: 07738. Mexico city, Mexico.}
\email{pmanriquem@ipn.mx}
\begin{document}
\begin{abstract}
In this manuscript, we study the limiting distribution  for the joint law of the largest and the smallest singular values 
for random circulant matrices
with generating sequence given by independent and identically distributed random elements satisfying the so-called
Lyapunov condition. 
Under  an appropriated normalization, the joint law of the extremal singular values
converges in distribution, as the matrix dimension tends to infinity, to an independent product of Rayleigh and Gumbel laws.
The latter implies that a normalized \textit{condition number} 
 converges in distribution to a Fr\'echet law as the dimension of the matrix increases.
\end{abstract}
\maketitle
\section{\textbf{Introduction}}
\markboth{{The asymptotic joint distribution of the largest and smallest singular values for random circulant matrices}}{{Limiting distribution of the condition number for random circulant matrices}}
\subsection{\textbf{Singular values and condition number}} 
The condition number was independently introduced by von Neumann and Goldstine in \cite{vonNG2,vonNG1}, and by Turing  in \cite{Turing}
 for studying the accuracy in the solution of a linear system in the presence of finite-precision arithmetic. 
Roughly speaking, the condition number measures how much the output value of a linear system can change by a small perturbation in the input argument, see \cite{Smale,Wozniakowski} for further details.

Let $\cA\in \mathbb{C}^{n\times m}$ be a matrix of dimension $n\times m$.
We denote the singular values of $\cA$ in non-decreasing order by $0\leq \sigma^{(n,m)}_{1}(\cA)\leq \cdots \leq \sigma^{(n,m)}_{n}(\cA)$. 
That is to say, they are the {square roots of the } eigenvalues of the $n$-square matrix $\cA\cA^*$, where $\cA^*$ denotes the conjugate transpose matrix of $\cA$.
 The condition number of $\cA$, $\kappa(\cA)$, is defined as
\begin{equation}\label{eq:cnA}
\kappa(\cA):=\frac{\sigma^{(n,m)}_{n}(\cA)}{ \sigma^{(n,m)}_{1}(\cA)}\quad \textrm{ whenever }\quad \sigma^{(n,m)}_{1}(\cA)>0.
\end{equation}
It is known that  
\[
\sigma^{(n,m)}_{1}(\cA)=\inf\{\|\mathcal{A}-\mathcal{B}\| : \mathcal{B}\in \mathfrak{S}^{(n,m)}\},
\]
where
\[
\mathfrak{S}^{(n,m)}:=\{\mathcal{B}\in \mathbb{C}^{n\times m}: \mathrm{rank}(\mathcal{B})<
\min\{n,m\}
\}
\]
and
$\|\cdot \|$ denotes the operator norm. For $n=m$, the smallest singular value
$\sigma^{(n,n)}_{1}(\cA)$ measures the distance between the
matrix $\cA$ and the set of  singular matrices. We refer to Chapter 1 in \cite{burgisser2013condition} for further details.

{The extremal singular values 
are ubiquitous in the applications and have produced a vast literature
in geometric functional analysis, mathematical physics, numerical linear algebra, time series, statistics, etc., see for instance 
\cite{arup2018,MikoshXie2016,edelmanphd,Rider2014,
Rudelson2010,TaoVu2010} and Chapter~5 in \cite{Vershynin2012}.}
{
The study of the largest and the smallest singular values has been very important in the study of sample correlation matrices, we refer to  \cite{HeinyMikosch2018,Heiny2017,
Heiny2021,Heiny2019} for further details.
Moreover, Poisson statistics for the largest eigenvalues in random matrices ensembles (Wigner ensembles for instance) are studied in \cite{Soshnikov2004,Soshnikov2006}.
}

Calculate or even estimate  the condition number of a generic matrix is a difficult task, see \cite{Sankar}.
{
In computational complexity theory, it is of interest to analyze the random condition number, that is, when the matrix $\mathcal{A}$ given in~\eqref{eq:cnA} is a random matrix.}
In \cite{edelman1988}, it is computed the limiting law of the condition number of a random rectangular matrix with independent and identically distributed (i.i.d. for short) standard Gaussian entries. Moreover,  the exact law of the condition number of a $2\times n$ matrix is  derived. In \cite{Wschebor2004},  for real square random matrices with i.i.d. standard Gaussian entries, no asymptotic lower and upper bounds for the tail probability of the condition number are established. {Later}, in \cite{Sutton2005}, the results are generalized for non-square matrices and analytic expressions  for the tail distribution of the condition number are obtained. 
Lower and upper bounds for the condition number (in $p$-norms) and the so-called average ``loss of precision" are studied in \cite{Szarek} for real and complex square random matrices with i.i.d. standard Gaussian entries.
In \cite{VisTre},  it is studied the case 
of random lower triangular matrices $\mathcal{L}_n$ of dimension $n$ with i.i.d. standard Gaussian random entries and shown that $(\kappa(\mathcal{L}_n))^{1/n}$ converges almost surely to $2$ as 
$n$ tends to infinity.
In \cite{Castillo}, using a Coulomb fluid technique,
it is derived asymptotics for the cumulative distribution function of the condition number for rectangular matrices with i.i.d. standard Gaussian entries.
More recently, 
distributional properties of the condition number for random matrices with i.i.d. Gaussian entries are established in \cite{AndersonWells,ChenDongarra,Shakil}
 {and large deviation asymptotics for condition numbers of sub-Gaussian distributions are given in \cite{Singull2021}.}
We recommend \cite{burgisser2013condition,cucker2016probabilistic,DemmelJ,edelmanphd} for a complete and {current} descriptions of condition number for random matrices.
\subsection{\textbf{Random circulant matrices}}
 {The (random) circulant matrices and (random) circulant type matrices are an important object in different areas of pure and applied mathematics, 
for instance 
compressed sensing,
cryptography,
discrete Fourier transform, {extreme value analysis,}
information processing,  machine learning, 
numerical analysis, {spectral analysis,}
time series analysis, etc. For more details we refer to \cite{aldrovandi,BoseGuha2011,BoseHazra2012,David2012,gray2006,
Simanca,RAU} and {the monograph on random circulant matrices \cite{arup2018}.}}
{
Some topics that have been studied are
spectral norms, extremal distributions, the so-called limiting spectral distribution for random circulant matrices and random circulant-type matrices and process convergence of fluctuations,
see
\cite{BoseHazra2011,BoseHazra22012,BoseHazra12011,
BoseHazra2010,BoseMaurya2020,BoseMitra2002,
BoseMitra20112,BoseSubhraSaha2009}.}

Let $\nu_0,\ldots,\nu_{n-1}\in \mathbb{C}$ be any given complex numbers.  We say that an $n\times n$ complex matrix $\textnormal{circ}(\nu_0,\ldots,\nu_{n-1})$ is circulant with generating elements $\{\nu_0,\ldots,\nu_{n-1}\}$ if it has the following structure{:}
\begin{equation*}\label{circulantes}
\textnormal{circ}(\nu_0,\ldots,\nu_{n-1}):=\left[\begin{array}{ccccc}
\nu_0 & \nu_{1} & \cdots & \nu_{n-2} & \nu_{n-1} \\
\nu_{n-1} & \nu_{0} & \cdots & \nu_{n-3} & \nu_{n-2} \\
\vdots & \vdots & \ddots & \vdots & \vdots \\
\nu_2 & \nu_3 & \cdots & \nu_0 & \nu_1\\
\nu_{1} & \nu_{2} & \cdots & \nu_{n-1} & \nu_{0}
\end{array}\right].
\end{equation*}
Let $\mathsf{i}$ be  the imaginary unit
and $\om_n=\exp(\mathsf{i}\cdot2\pi/n)$  be a primitive $n$-th root of unity. Define the Fourier unitary matrix of order $n$ by
$\cF_n:=\frac{1}{\sqrt{n}}(\om^{kj}_n)_{k,j\in \{0,\ldots,n-1\}}$.
It is well-known that $\textnormal{circ}(\nu_0,\ldots,\nu_{n-1})$ can be diagonalized  as follows:
\begin{equation}\label{eq:decomposition}
\textnormal{circ}(\nu_0,\ldots,\nu_{n-1})=\cF^*_n \cD_n \cF_n,
\end{equation}
where  $\cD_n:=\textnormal{diag}(\lambda^{(n)}_1,\ldots,\lambda^{(n)}_n)$ is a diagonal matrix with entries satisfying 
\begin{equation*}
\la^{(n)}_k=\sum_{j=0}^{n-1}\nu_j \,\om^{kj}_n\quad \textrm{ for }\quad  k\in \{0,\ldots,n-1\}.
\end{equation*}
By~\eqref{eq:decomposition} we note that
$(\la^{(n)}_k)_{k\in \{0,\ldots,n-1\}}$
are the eigenvalues of  $\textnormal{circ}(\nu_0,\ldots,\nu_{n-1})$.
For $\xi_0,\ldots,\xi_{n-1}$ being  random elements, we say that 
$\mathcal{C}_n:=\textnormal{circ}(\xi_0,\ldots,\xi_{n-1})$ is an $n\times n$ random circulant matrix. Then
its eigenvalues are given by 
\begin{equation}\label{eq:eigenvals}
\la^{(n)}_k=\sum\limits_{j=0}^{n-1}\xi_j\, \om^{kj}_n\quad \textrm{ for }\quad k\in \{0,\ldots,n-1\}.
\end{equation}
Since any circulant matrix is a normal matrix, its singular values are given by 
$(|\la^{(n)}_k|)_{k\in \{0,\ldots,n-1\}}$,
where the symbol $|\cdot|$ denotes the complex modulus.
Then the extremal singular values of $\mathcal{C}_n$ can be written as
\begin{equation}\label{eq:notation}
\sigma^{(n)}_{\min}:=\min_{k\in \{0,\ldots,n-1\}}\big|\la^{(n)}_k\big|
\quad \textrm{ and } \quad 
\sigma^{(n)}_{\max}:=\max_{k\in \{0,\ldots,n-1\}}\big|\la^{(n)}_k\big|.
\end{equation}
By~\eqref{eq:cnA} and~\eqref{eq:notation} it follows that the {random} condition number of $\mathcal{C}_n$ 
is given by 
\begin{equation}\label{eq:defk}
\kappa(\mathcal{C}_n)=\frac{\sigma^{(n)}_{\max}}{\sigma^{(n)}_{\min}} \quad \textrm{ whenever }\quad \sigma^{(n)}_{\min}>0.
\end{equation}
We stress that the random variables $\sigma^{(n)}_{\max}$ and $\sigma^{(n)}_{\min}$ are not independent.
Thus, it is a {\it challenging problem} to compute (or estimate) the distribution of condition numbers of
 circulant matrices for general  i.i.d. random coefficients $\xi_0,\ldots,\xi_{n-1}$.
\subsection{\textbf{Main results}}
The problem of computing the limiting distribution of the condition number for square matrices with
i.i.d. (real or complex) standard Gaussian random entries  is completely analyzed in Chapter 7 of \cite{edelmanphd}.
In this manuscript we focus on the computation of the limiting distribution of  $\kappa(\mathcal{C}_n)$  for $\xi_0,\ldots,\xi_{n-1}$ being i.i.d. real random variables satisfying the so-called Lyapunov condition,
see Hypothesis~\ref{hyp:lc} below.
In fact, the limiting distribution is a Fr\'echet distribution that belongs to the class of the so-called extreme value distributions \cite{GA}.
Non-asymptotic estimates for the condition number for random
circulant and Toeplitz matrices with i.i.d. standard Gaussian random entries
are given in \cite{pan2001structured,panSZ2015}. 
 The approach and results of
\cite{pan2001structured,panSZ2015}  
 are different in nature from our results given in  Theorem~\ref{thm:generalcase}.

We assume the following integrability condition. It appears for instance in the so-called Lyapunov Central Limit Theorem, see Section 7.3.1 in \cite{Ash}.
{
Along this manuscript, the set of non-negative integers is denoted by $\mathbb{N}_0$.}
\begin{hyp}[Lyapunov condition]\label{hyp:lc}
We assume that $(\xi_j)_{j\in \mathbb{N}_0}$ is a sequence
of  i.i.d. non-degenerate real random variables on some probability space $(\Omega,\mathcal{F}, \mathbb{P})$
with zero mean and unit variance. If there exists $\delta>0$ such that
$
\mathbb{E}\left[|\xi_0|^{2+\delta}\right]<\infty$,
where $\mathbb{E}$ denotes the expectation with respect to $\mathbb{P}$, we say that $(\xi_j)_{j\in \mathbb{N}_0}$ satisfies the Lyapunov integrability condition.
\end{hyp}

{We note that a sequence of i.i.d. non-degenerate sub-Gaussian random variables satisfies the Lyapunov condition.}
Before state the main result and its consequences, we introduce some notation.
For shorthand and in a conscious abuse of notation, we use indistinctly the following notations for the exponential function:
$\exp(a)$ or $e^{a}$ for $a\in \mathbb{R}$.
We denote by
$\ln(\cdot)$ the Napierian logarithm function and we use the same notation $|\cdot|$ for the complex modulus and the absolute value.

The main result of this manuscript is the following.
\begin{thm}[Joint asymptotic distribution of the smallest and the largest singular values]\label{thm:generalcase}
Assume that Hypothesis~\ref{hyp:lc} is valid.
Then it follows that
\begin{equation}\label{eq:RGlimite}
\lim\limits_{n\to \infty}\Pr{\sigma^{(n)}_{\min}\leq  x,\frac{\sigma^{(n)}_{\max}-\fa_n}{\fb_n}\leq y}=R(x)G(y)
\quad \textrm{ for any } x\geq 0 \textrm{ and } y\in \mathbb{R},
\end{equation}
where 
\begin{equation}\label{eq:defRG}
R(x)=1-\exp\PP{-x^2/2},~x\geq 0,\quad \textrm{ and }\quad G(y)=\exp\PP{-e^{-y}},~ y\in \mathbb{R}
\end{equation}
are the  Rayleigh distribution and Gumbel distribution, respectively,
and the normalizing constants are given by
\begin{equation}\label{eq:defab}
\fa_n=\sqrt{n\ln(n/2)}\quad \textrm{ and } \quad \fb_n=\frac{1}{2}\sqrt{\frac{n}{\ln (n/2)}},\quad n\geq 3.
\end{equation}
\end{thm}
The proof of Theorem~\ref{thm:generalcase}
is based in the {Davis--Mikosch} method used to prove that a normalized  maximum of the periodogram converges in distribution to the Gumbel law, see
Theorem~2.1 in \cite{DavisMikosh}.
This method relies on Einmahl's multivariate extension of the so-called Koml\'os--Major--Tusn\'ady approximation, see \cite{Einmahl}.
However, the laws of the largest singular value $\sigma^{(n)}_{\max}$ and the smallest singular value $\sigma^{(n)}_{\min}$ are strongly correlated, and hence the computation of the condition number law is a priori hard. Thus, using the {Davis--Mikosch} method we compute the limiting law of 
the random vector
\begin{equation}\label{eq:jointvc}
\left(\sigma^{(n)}_{\min},\frac{\sigma^{(n)}_{\max}-\fa_n}{\fb_n}\right)
\end{equation}
whose components are not independent.
Applying the Continuous Mapping Theorem we deduce the limiting law of the condition number $\kappa(\mathcal{C}_n)$, see Corollary~\ref{cor:A} item (iv).
{Along the lines of~\cite{HeinyYslas2021}, one can obtain
convergence of the point processes of the singular values in the setting of Theorem~\ref{thm:generalcase}.}
We remark that~\eqref{eq:eigenvals} resembles the discrete Fourier transform and it is related with the so-called periodogram, which 
have been used in many areas of apply science.
The maximum of the periodogram has been already studied {for instance in \cite{DavisMikosh,Kokoszka2000,Lin2009,Turkman1984} } and under a suitable rescaling, the Gumbel distribution appears as a limiting law. The asymptotic behavior of Fourier transforms of stationary
ergodic sequences with finite second moments is analyzed and shown that asymptotically the real part and the imaginary part of the Fourier transform decouple in a product of independent Gaussian distributions, see \cite{Peligrad2010}.
{
The so-called quenched central limit theorem for the discrete Fourier transform of a stationary ergodic process is obtained in \cite{Barrera2016} and the central limit theorem for discrete Fourier transforms of function times series are given in  \cite{Cerovecki2017}. In addition, in \cite{Siegfried2020} the maximum of the periodogram is studied in times series with values in a Hilbert space.} 
However, up to our knowledge, 
Theorem~\ref{thm:generalcase} is not immediately implication of these results.
As a consequence of Theorem~\ref{thm:generalcase} we obtain the limiting distribution of the (normalized) largest singular value, the smallest singular value and the (normalized) condition number
 as Corollary~\ref{cor:A} states.
{
For a  $n\times n$ symmetric random Toeplitz matrix satisfying Hypothesis~\ref{hyp:lc}, it is shown in \cite{SenVirag} that the largest eigenvalue scaled by $\sqrt{2n\ln(n)}$ converges in $L^{2+\delta}$ as $n\to \infty$ to the constant $\|\textsf{Sin}\|^2_{2\to 4}=0.8288\ldots$.}

We point out that
Theorem~1 in \cite{BrycSethuraman} yields
 that the Gumbel distribution is the limiting distribution for  the 
(renormalized) largest singular value for symmetric circulant random matrices with generating i.i.d. sequence (half of its entries) satisfying Hypothesis~\ref{hyp:lc}.
Also, for suitable normalization the limiting law of the largest singular value of Hermitian Gaussian circulant matrices has Gumbel distribution, see Corollary~5 in \cite{Meckes}.

Recall that the square root of a 
Exponential distribution with parameter $\lambda$ has 
Rayleigh distribution with parameter $(2\lambda)^{-1/2}$. The exponential law  appears as the limiting distribution of the minimum modulus of  trigonometric polynomials, see Theorem~1 in \cite{Yakir2020} for the Gaussian case and  Theorem~1.2 of \cite{Cook2021} for the sub-Gaussian case. 

{
The Fr\'echet distribution appears as a limiting distribution of the largest eigenvalue (rescaled) for random real symmetric matrices with independent and heavy tailed entries, see
Corollary~1 in \cite{Aunger2009} {and \cite{Bojan2021} for the non-i.i.d. case.}}
The Fr\'echet distribution in 
Corollary~\ref{cor:A} item (iv) has cumulative distribution $F(t)=\exp(-t^{-2})
\Ind{t>0}$
with shape parameter $2$, scale parameter $1$ and location parameter $0$.
Moreover, it does not possess finite variance.
For descriptions of extreme value distributions and limiting theorems we refer to \cite{GA}.
\begin{cor}[Asymptotic distribution of the largest singular value, the smallest singular value and the condition number]\label{cor:A}
Let the notation and hypothesis of 
Theorem~\ref{thm:generalcase} be valid. The following holds.
\begin{itemize}
\item[(i)] The normalized maximum $\frac{\sigma^{(n)}_{\max}-\fa_n}{\fb_n}$ and the minimum
$\sigma^{(n)}_{\min}$ are asymptotically independent.
\item[(ii)] The normalized maximum $\frac{\sigma^{(n)}_{\max}-\fa_n}{\fb_n}$ converges in distribution as $n\to \infty$ to a Gumbel distribution, i.e.,
\[
\lim\limits_{n\to \infty}\Pr{\frac{\sigma^{(n)}_{\max}-\fa_n}{\fb_n}\leq y}=G(y)
\quad \textrm{ for any }\quad y\in \mathbb{R}.
\]
\item[(iii)] The minimum $\sigma^{(n)}_{\min}$ converges in distribution as $n\to \infty$ to a Rayleigh distribution, i.e.,
\[
\lim\limits_{n\to \infty}\Pr{\sigma^{(n)}_{\min}\leq  x}=R(x)
\quad \textrm{ for any }\quad x\geq 0.
\]
\item[(iv)]
The condition number $\kappa(\mathcal{C}_n)$  converges in distribution as $n\to \infty$ to a Fr\'echet distribution, i.e.,
\[
\lim\limits_{n\to \infty}
\Pr{\frac{\kappa(\mathcal{C}_n)}{\sqrt{\frac{1}{2}n\ln(n)}}\leq  z} =F(z)\quad \textrm{ for any }\quad z>0, 
\]
where 
$F(z)=\exp\left(-z^{-2}\right)\Ind{z>0}$.
\end{itemize}
\end{cor}
{
In the sequel, we briefly compare our results with the literature about the limiting law of the condition number, the smallest singular value and the largest singular value.}
{
\begin{rem}[Frechet's distributions as limiting distributions for condition numbers]
The Fr\'echet distribution with shape parameter $2$, scale parameter $2$ and location parameter $0$ is the limiting distribution as the dimension growths of $\kappa(\mathcal{A}_n)/n$, where $\mathcal{A}_n$ is a square matrix of dimension $n$ with i.i.d. complex Gaussian entries, see
Theorem~6.2 in \cite{edelman1988}. When $\mathcal{A}_n$ has real i.i.d. Gaussian entries, the limiting law of $\kappa(\mathcal{A}_n)/n$ converges   in distribution as $n$ tends to infinity to a random variable with an explicit density, see Theorem~6.1 in \cite{edelman1988}. We stress that such density does not belong to the Fr\'echet family of distributions.
We also point out that 
the distribution of the so-called Demmel condition for (real and complex) Wishart matrices are given explicitly in \cite{Edel1992}.
\end{rem}
}

{
\begin{rem}[A word about the smallest singular value $\sigma_1(\mathcal{A}_n)$]
The behavior of the smallest singular value appears naturaly in numerical inversion of large matrices. 
For instance, when the random matrix $\mathcal{A}_n$ has complex i.i.d. Gaussian entries, for all $n$ the random variable of $n\sigma_1(\mathcal{A}_n)$ has the Chi-square distribution with two degrees of freedom, see Corollary~3.3 in \cite{edelman1988}.
For  a random matrix $\mathcal{A}_n$ with real i.i.d. Gaussian entries, 
it is shown in Corollary~3.1 in \cite{edelman1988} that
$n\sigma_1(\mathcal{A}_n)$
converges in distribution as the dimension increases to a random variable with an explicit density.
For further discussion about the smallest singular values, we refer to
\cite{BaiYin993,BoseHachem2020,
Gregoratti2021,HeinyMikosch2018,
HuangTik2020,Kostlan,TaoVu2010,Tatarko}.
\end{rem}
 }

{
\begin{rem}[A word about the largest singular value $\sigma_n(\mathcal{A}_n)$]
A lot is known about the behavior of the largest singular value.
As an illustration, 
for  a random matrix $\mathcal{A}_n$ with real i.i.d. Gaussian entries, 
it is shown in Lemma~4.1 in \cite{edelman1988} that
$(1/n)\sigma_n(\mathcal{A}_n)$
converges in probability to $4$  as $n$ growths while 
for the random matrix $\mathcal{A}_n$ with complex i.i.d. Gaussian entries,
 $(1/n)\sigma_n(\mathcal{A}_n)$
converges in probability to $8$.
We stress that the Gumbel distribution is the
limiting law of the spectral radius of Ginibre ensembles, see Theorem~1.1 in \cite{Rider2014}.
Recently, it is shown in Theorem~4 in \cite{ArenasAbreu} that the  Gumbel distribution is the limiting law of the largest eigenvalue of a Gaussian Laplacian matrix. 
For further discussion, we recommend to 
\cite{Aunger2009,BaiSilver988,
BrycSethuraman,Siegfried2020,
DavisMikosh,HeinyMikosch2018,Meckes,
Soshnikov2004,Soshnikov2006}.
\end{rem}
}

{We continue with the proof of 
Corollary~\ref{cor:A}.}
\begin{proof}[Proof of Corollary~\ref{cor:A}]
Items (i)-(iii) follow directly from 
Theorem~\ref{thm:generalcase} and the Continuous Mapping Theorem (see Theorem~13.25 in \cite{Klenke}). 
In the sequel, we prove Item (iv).
By~\eqref{eq:defk} we have
\begin{equation}\label{eq:split}
\frac{\kappa(\mathcal{C}_n)}{\sqrt{\frac{1}{2}n\ln(n)}}=\frac{\fb_n}{{\sqrt{\frac{1}{2}n\ln(n)}}}
\frac{1}{\sigma^{(n)}_{\min}}
\frac{\sigma^{(n)}_{\max}-\fa_n}{\fb_n}+
\frac{\fa_n}{{\sqrt{\frac{1}{2}n\ln(n)}}}\frac{1}{\sigma^{(n)}_{\min}}.
\end{equation}
Theorem~\ref{thm:generalcase}
with the help of the Continuous Mapping Theorem, limits
\[
\lim\limits_{n\to \infty}\frac{\fb_n}{{\sqrt{\frac{1}{2}n\ln(n)}}}=0,\qquad 
\lim\limits_{n\to \infty}
\frac{\fa_n}{{\sqrt{\frac{1}{2}n\ln(n)}}}=\sqrt{2},
\]
and the Slutsky Theorem (see Theorem~13.18 in \cite{Klenke})
implies 
\begin{equation}\label{eq:limitdist}
\begin{split}
\frac{\fb_n}{{\sqrt{\frac{1}{2}n\ln(n)}}}\frac{1}{\sigma^{(n)}_{\min}}
\frac{\sigma^{(n)}_{\max}-\fa_n}{\fb_n}&\longrightarrow 0\cdot \frac{\mathfrak{G}}{\mathfrak{R}}=0, \quad \textrm{ in distribution, } \quad \textrm{ as }\quad n\to \infty,\\
\frac{\fa_n}{{\sqrt{\frac{1}{2}n\ln(n)}}}\frac{1}{\sigma^{(n)}_{\min}}&\longrightarrow \frac{\sqrt{2}}{\mathfrak{R}},
\quad \textrm{ in distribution, } \quad \textrm{ as }\quad n\to \infty,
\end{split}
\end{equation}
where $\mathfrak{G}$ and $\mathfrak{R}$ are random variables with Gumbel and Rayleigh distributions, respectively.
The Slutsky Theorem
with the help of~\eqref{eq:split} and~\eqref{eq:limitdist} implies
\begin{equation}
\frac{\kappa(\mathcal{C}_n)}{\sqrt{\frac{1}{2}n\ln(n)}}\longrightarrow \frac{\sqrt{2}}{\mathfrak{R}},
\quad \textrm{ in distribution, } \quad \textrm{ as }\quad n\to \infty.
\end{equation}
Lemma~\ref{lem:1Ray} in Appendix~\ref{ap:tools} implies (iv) that 
$\frac{\sqrt{2}}{\mathfrak{R}}$ possesses Fr\'echet distribution $F$.
This finishes the proof of Item (iv).
\end{proof}
Recently, the condition number for powers of matrices
has been studied  in \cite{HuangTik2020}.
As a consequence of 
Theorem~\ref{thm:generalcase} we have the following corollary which, 
in particular, gives the limiting distribution of the condition number for the powers of $\mathcal{C}_n$.
\begin{cor}[Asymptotic distribution of $p$-th power of the maximum, the minimum and the condition number]\label{cor:B}
Let the notation and hypothesis of 
Theorem~\ref{thm:generalcase} be valid and take $p\in \mathbb{N}$.
The following holds.
\begin{itemize}
\item[(i)] Asymptotic distribution of the $p$-th power of the normalized maximum. For any $y\in \mathbb{R}$ it follows that
\[
\lim\limits_{n\to \infty}\Pr{\frac{\sigma^{(n)}_{\max}(\mathcal{C}^p_n)-A_n(p)}{B_n(p)}\leq y}=G\left(y\right),
\]
where
$A_n(p)=\fa^p_n$ and
$B_n(p)=p\fb^p_n(2^p-1)$ for all $n\geq 3$.
\item[(ii)] Asymptotic distribution of the $p$-th power of the minimum. For any $x\geq 0$ it follows that
\[
\lim\limits_{n\to \infty}\Pr{\sigma^{(n)}_{\min}(\mathcal{C}^p_n)\leq  x}=R\left(x^{1/p}\right).
\]
\item[(iii)]
Asymptotic distribution of the $p$-th power of the condition number. For any $z>0$ it follows that
\[
\lim\limits_{n\to \infty}
\Pr{\frac{\kappa(\mathcal{C}^p_n)}{(\frac{1}{2}n\ln(n))^{p/2}}\leq  z} =F\left(z^{1/p}\right).
\]
\end{itemize}
\end{cor}
\begin{proof}
By~\eqref{eq:decomposition}
we have  $\mathcal{C}^p_n=\cF^*_n \cD^p_n \cF_n$ for any $p\in \mathbb{N}$. Therefore, $\mathcal{C}^p_n$ is a  normal matrix and then
\[
\sigma^{(n)}_{\min}(\mathcal{C}^p_n)=(\sigma^{(n)}_{\min}(\mathcal{C}_n))^p,\quad
\sigma^{(n)}_{\max}(\mathcal{C}^p_n)=(\sigma^{(n)}_{\max}(\mathcal{C}_n))^p
\quad
\textrm{ and } \quad
\kappa(\mathcal{C}^p_n)=(\kappa(\mathcal{C}_n))^p.
\]
Consequently, the statements (ii)-(iii) are  direct consequences of 
Theorem~\ref{thm:generalcase} with the help of 
Lemma~\ref{lem:doublelimit} in Appendix~\ref{ap:tools} and the Continuous Mapping Theorem.

In the sequel, we prove (i).
For $p=1$, item (i) follows directly from item (ii) of Corollary~\ref{cor:A}. For any $p\in \mathbb{N}\setminus \{1\}$
set
\[
A_n:=A_n(p)=\fa^p_n\quad \textrm{ and } \quad
B_n=p\sum\limits_{j=0}^{p-1} \fa^j_n \fb^{p-j}_n.
\]
Let $y\in \mathbb{R}$ and observe that 
$B_n y+A_n\to \infty$ as $n\to \infty$ due to
$\fb_n/ \fa_n\to 0$ as $n\to \infty$.
Then for $n$ large enough we have 
\begin{equation}\label{eq:GGG}
\begin{split}
\Pr{\frac{\sigma^{(n)}_{\max}(\mathcal{C}^p_n)-A_n}{B_n}\leq y}&=\Pr{\sigma^{(n)}_{\max}(\mathcal{C}_n)\leq (B_n y+A_n)^{1/p}}\\
&{=}\Pr{\frac{\sigma^{(n)}_{\max}(\mathcal{C}_n)-\fa_n}{\fb_n}\leq \frac{(B_n y+A_n)^{1/p}-\fa_n}{\fb_n}}.
\end{split}
\end{equation}
We claim that
\begin{equation}
\lim\limits_{n\to \infty}\frac{(B_n y+A_n)^{1/p}-\fa_n}{\fb_n}=y\quad \textrm{ for any }\quad y\in \mathbb{R}.
\end{equation} 
Indeed, notice that 
\begin{equation}
\begin{split}
\frac{(B_n y+A_n)^{1/p}-\fa_n}{\fb_n}&=\frac{\fa_n}{\fb_n}\left(\left(p\sum\limits_{j=0}^{p-1} \fa^{j-p}_n \fb^{p-j}_n y+1\right)^{1/p}-1\right)\\
&=\frac{1}{\delta_n}\left(\left(py\sum\limits_{j=1}^{p} \delta^{j}_n +1\right)^{1/p}-1\right)=\frac{1}{\delta_n}\left(\left(py\sum\limits_{j=1}^{p} \delta^{j}_n +1\right)^{1/p}-1\right),
\end{split}
\end{equation}
where $\delta_n:=\fb_n/ \fa_n$. 
Since 
$\delta_n\to 0$ as $n\to \infty$, L'H\^opital's rule yields
\begin{equation}
\begin{split}
\lim\limits_{n\to \infty}\frac{1}{\delta_n}\left(\left(py\sum\limits_{j=1}^{p} \delta^{j}_n +1\right)^{1/p}-1\right)
=\lim\limits_{n\to \infty} \left[\frac{1}{p}\left(py\sum\limits_{j=1}^{p} \delta^{j}_n +1\right)^{1/p-1}py \sum\limits_{j=1}^{p} j\delta^{j-1}_n\right]=y.
\end{split}
\end{equation}
Hence,~\eqref{eq:GGG} with the help of 
Lemma~\ref{lem:doublelimit} in Appendix~\ref{ap:tools} implies
\[
\lim\limits_{n\to \infty}\Pr{\frac{\sigma^{(n)}_{\max}(\mathcal{C}^p_n)-A_n}{B_n}\leq y}=G(y)\quad \textrm{ for any }\quad y\in\mathbb{R}.
\]
By~\eqref{eq:defab},
a straightforward computation yields
$B_n=B_n(p)=p\fb^p_n(2^p-1)$ for all $n\geq 3$.
This finishes the proof of item (i).
\end{proof}
The following proposition states the standard Gaussian case. It is an immediately consequence of 
Theorem~\ref{thm:generalcase}. However, since the random variables are i.i.d. with standard Gaussian distribution, 
the computation of 
the law of~\eqref{eq:jointvc} and its limiting distribution
 can be carried out explicitly.
  We decide to include an alternative and instructive proof in order to illustrate how  {we identify} the right-hand side of~\eqref{eq:RGlimite}. 
{
\begin{prop}[Joint asymptotic distribution of the minimum and the maximum: the Gaussian case]\label{thm:Gaussiancase}
Let $(\xi_j)_{j\in \mathbb{N}_0}$ be a sequence of  i.i.d. real Gaussian coefficients with zero mean and unit variance. Then it follows that
\begin{equation}\label{eq:prolimit}
\lim\limits_{n\to \infty}\Pr{\sigma^{(n)}_{\min}\leq  x,\frac{\sigma^{(n)}_{\max}-\fa_n}{\fb_n}\leq y}=R(x)G(y)
\quad \textrm{ for any } x\geq 0 \textrm{ and } y\in \mathbb{R},
\end{equation}
where the functions $R$ and $G$
are given in~\eqref{eq:defRG},
 and the normalizing constants $\fa_n$, $\fb_n$ are defined in~\eqref{eq:defab}.
\end{prop}
\begin{proof}
Let $q_n:=\lfloor n/2\rfloor$. For $1\leq k\leq q_n$ we note that 
$\lambda^{(n)}_k = \overline{\lambda^{(n)}_{n-k}}$, where $\overline{\lambda}$ denotes the complex conjugate of  a complex number $\lambda$.
Then we have
\begin{equation}
\sigma^{(n)}_{\max} = \max\LL{\abs{\lambda^{(n)}_k} : 0\leq k\leq q_n}
\quad \textrm{ and } \quad
\sigma^{(n)}_{\min} =  \min\LL{\abs{\lambda^{(n)}_k} : 0\leq k\leq q_n}. 
\end{equation}
The complex modulus of $\lambda^{(n)}_k$ is given by 
 $\abs{\lambda^{(n)}_{k}}=\sqrt{c_{k,n}^2+s_{k,n}^2}$, where 
$c_{k,n} := \sum_{j=0}^{n-1} \xi_j \cos\PP{2\pi\frac{kj}{n}}$ and $ s_{k,n}:= \sum_{j=0}^{n-1} \xi_j \sin\PP{2\pi\frac{kj}{n}}$ for  $0\leq k \leq q_n$.
Note that $s_{0,n}=0$.
By 
straightforward computations  we obtain for any $n\in \mathbb{N}$
\begin{itemize}
\item[(i)] $\E{c_{k,n}}=\E{s_{k,n}}=0$\;\; for $0\leq k\leq q_n$,
\item[(ii)] $\E{c^2_{0,n}}=n$, $\E{c_{k,n}^2}=\E{s_{k,n}^2}=\frac{n}{2}$\;\; for $1\leq k\leq q_n$,
\item[(iii)] $\E{c_{k,n}\cdot s_{l,n}}=\E{c_{k,n}\cdot c_{l,n}}=\E{s_{k,n}\cdot s_{l,n}}= 0$\;\; for $0\leq l<k\leq q_n$.
\end{itemize}
Then (i), (ii) and (iii) implies that $\frac{1}{\sqrt{n}}\PP{c_{0,n},c_{1,n},s_{1,n},\ldots,c_{q_n,n},s_{q_n,n}}$ is a Gaussian vector such that its entries are not correlated, i.e., it has independent Gaussian entries. Thus, the random vector
$
\frac{1}{n}\big(c_{0,n}^2,c^2_{1,n}+s^2_{1,n},\ldots,c^2_{q_n,n}+s^2_{q_n,n}\big)
=\frac{1}{n}\big(|\lambda^{(n)}_{0}|^2,|\lambda^{(n)}_{1}|^2,\ldots,|\lambda^{(n)}_{q_n}|^2\big)
$  has independent random entries satisfying
\begin{equation}\label{eq:distribution9}
\begin{cases}
\frac{1}{n}|\lambda^{(n)}_{0}|^2 & \textrm{has a Chi-square distribution $\chi_1$ with one-degree of freedom},\\
\frac{1}{n}|\lambda^{(n)}_{k}|^2 & \textrm{has an exponential distribution $\mathsf{E}_1$ with parameter one for } 1\leq k\leq q_n-1
\end{cases}
\end{equation}
and 
\begin{equation}\label{eq:distribution99}
\begin{cases}
\frac{1}{n}|\lambda^{(n)}_{q_n}|^2 & 
\textrm{has $\chi_1$ distribution for $n$ being an even number (due to $s_{q_n,n}=0$)},\\
\frac{1}{n}|\lambda^{(n)}_{q_n}|^2 & 
\textrm{has $\mathsf{E}_1$ distribution for $n$ being an odd number}. 
\end{cases}
\end{equation}
For $n$ being an odd number,~\eqref{eq:distribution9},~\eqref{eq:distribution99} and Lemma~\ref{lem:080720201357} in 
Appendix~\ref{ap:tools} imply
\begin{equation}\label{eq:prodct}
\begin{split}
\Pr{\sigma^{(n)}_{\min}\leq  x,\sigma^{(n)}_{\max}\leq \fb_n y+\fa_n}&=\prod_{j=0}^{q_n} \Pr{|\lambda^{(n)}_j|\leq \fb_n y+\fa_n }-
\prod_{j=0}^{q_n} \Pr{x<|\lambda^{(n)}_j|\leq \fb_n y+\fa_n }\\
&\hspace{-2cm}=\prod_{j=0}^{q_n} \Pr{
\frac{|\lambda^{(n)}_j|^2}{n}\leq \frac{(\fb_n y+\fa_n)^2}{n} }-
\prod_{j=0}^{q_n} \Pr{\frac{x^2}{n}<
\frac{|\lambda^{(n)}_j|^2}{n}\leq \frac{(\fb_n y+\fa_n)^2}{n} }\\
&\hspace{-2cm}=\Pr{
\chi_1\leq \frac{(\fb_n y+\fa_n)^2}{n} }\left(\Pr{
\textsf{E}_1\leq \frac{(\fb_n y+\fa_n)^2}{n} }\right)^{q_n}\\
&\hspace{-2cm}\qquad-
\Pr{\frac{x^2}{n}<\chi_1\leq \frac{(\fb_n y+\fa_n)^2}{n} }\left(\Pr{\frac{x^2}{n}<
\textsf{E}_1\leq \frac{(\fb_n y+\fa_n)^2}{n} }\right)^{q_n}.
\end{split}
\end{equation}
We observe that 
$
\lim\limits_{n\to \infty} \frac{x^2}{n}=0$ and $
\lim\limits_{n\to \infty}\frac{(\fb_n y+\fa_n)^2}{n}=\infty$ for any $x,y\in \mathbb{R}$.
Recall that for any non-negative numbers $u$ and $v$ such that $u<v$ it follows that
$
\Pr{u<\textsf{E}_1\leq v}=\exp(-u)-\exp(-v)$.
Then,
a straightforward calculation yields
\begin{equation}
\begin{split}
\left(\Pr{\frac{x^2}{n}<
\textsf{E}_1\leq \frac{(\fb_n y+\fa_n)^2}{n} }\right)^{q_n}
=\exp\left(-\frac{x^2}{2}\right)
\left(1- 
\frac{\exp\left(\frac{x^2}{n}-\frac{y^2}{4\ln(n/2)}\right)\exp(-y)}{n/2}
\right)^{n/2-1/2}.
\end{split}
\end{equation}
Hence, for any $x,y$ it follows that
\[
\lim\limits_{n\to \infty}\left(\Pr{\frac{x^2}{n}<
\textsf{E}_1\leq \frac{(\fb_n y+\fa_n)^2}{n} }\right)^{q_n}=\exp\left(-\frac{x^2}{2}\right)\exp\left(-e^{-y}\right).
\]
The preceding limit with the help of~\eqref{eq:prodct} implies limit~\eqref{eq:prolimit}.
Similar reasoning yields the proof when $n$ is an even number. 
This finishes the proof 
of Proposition~\ref{thm:Gaussiancase}.
\end{proof}
}
The rest of the manuscript is organized as follows.
Section~\ref{sec:approx} is divided in five subsections.
In Subsection~\ref{sub:removing} 
we prove that the asymptotic behavior of $|\lambda^{(n)}_0|$ can be removed from our computations.
In Subsection~\ref{sec:bounded} we establish that the original sequence of random variables can be assumed to be bounded.
In Subsection~\ref{sec:smooth} we provide a procedure in which  we smooth (by a Gaussian perturbation) the bounded sequence obtained in Subsection~\ref{sec:bounded}.
In Subsection~\ref{sec:boundedsmooth} we prove the main result for the bounded and smooth sequence given in Subsection~\ref{sec:smooth}.
Finally, in Section~\ref{sec:prueba} we summarize all the results proved in previous subsections and show  
Theorem~\ref{thm:generalcase}. 
{
Finally, there is Appendix~\ref{ap:tools}
 which collects
technical results used in the main text.
}
\section{\textbf{Koml\'os--Major--Tusn\'ady approximation}}\label{sec:approx}
In this section, we show that
Theorem~\ref{thm:generalcase} can be {deduced} by a Gaussian approximation in which computations can be carried out. 
Roughly speaking, we show that
\begin{equation}\label{eq:reduction9}
\Pr{\sigma^{(n)}_{\min}\leq  x,\; \frac{\sigma^{(n)}_{\max}-\fa_n}{\fb_n}\leq y}  - 
   \Pr{{\sigma}^{(n,\NN)}_{\min} \leq \sqrt{\frac{2}{n}} x,\; \sigma^{(n,\NN)}_{\max}\leq  \sqrt{\frac{2}{n}}(\fb_n y+\fa_n) }=\textnormal{o}_n(1),
\end{equation}
where $\textnormal{o}_n(1)\to 0$ as $n\to \infty$, $(\fa_n)_{n\in \mathbb{N}}$ and $(\fb_n)_{n\in \mathbb{N}}$ are the sequences defined in~\eqref{eq:defab} in 
Theorem~\ref{thm:generalcase}, and 
$\sigma^{(n,\NN)}_{\min}$ and $\sigma^{(n,\NN)}_{\max}$ denote the smallest and the largest singular values, respectively, 
of a random circulant matrix with generating sequence given by i.i.d. bounded and smooth random variables which
can be well-approximated by a standard Gaussian distribution.

Along this section, for any set $A\subset \Omega$, we denote its complement with respect to $\Omega$ by $A^{\textnormal{c}}$.
We also point out the following immediately relation:
\begin{equation}\label{eq:equiv}
\begin{split}
&\Pr{\sigma^{(n)}_{\min}\leq  x,\; \frac{\sigma^{(n)}_{\max}-\fa_n}{\fb_n}\leq y}\\
&\hspace{3cm}= \Pr{ \bigcap_{k=0}^{n-1}\LL{\abs{\lambda^{(n)}_k}\leq \fb_n y + \fa_n}} - \Pr{\bigcap_{k=0}^{n-1}\LL{x < \abs{\lambda^{(n)}_k}\leq \fb_n y + \fa_n}}.
\end{split}
\end{equation}
\subsection{\textbf{Removing the singular value $|\lambda^{(n)}_0|$}}\label{sub:removing}
In this subsection we show that $\left|\lambda^{(n)}_0\right|$ can be removed from our computations.
In other words, we only need to consider our computations over the array $\PP{\abs{\lambda^{(n)}_k}}_{k\in \{1,\ldots,n-1\}}$ as the following lemma states.
\begin{lem}[Asymptotic behavior of $|\lambda^{(n)}_0|$ is negligible]\label{lem:DifZero070520291540}
Assume that Hypothesis~\ref{hyp:lc} is valid.
Then for any $x,y\in\R$ it follows that
\begin{equation}\label{eq:limitremoving}
\lim\limits_{n\to \infty}
\left|
\Pr{\bigcap_{k=1}^{n-1} \LL{x < \abs{\lambda^{(n)}_k}\leq \fb_n y + \fa_n}} - \Pr{\bigcap_{k=0}^{n-1} \LL{x < \abs{{\lambda}^{(n)}_k}\leq \fb_n y + \fa_n}}\right|=0.
\end{equation}
\end{lem}
\begin{proof} 
For any $n\in \mathbb{N}$ we set 
\[
A:=\bigcap_{k=0}^{n-1} \LL{x < \abs{\lambda^{(n)}_k}\leq \fb_n y + \fa_n}\quad \textrm{ and } \quad B:=\bigcap_{k=1}^{n-1} \LL{x < \abs{{\lambda}^{(n)}_k}\leq \fb_n y + \fa_n}.
\] 
Since ${A}\subset {B}$, we have
\begin{equation}\label{eq:removing1}
\begin{split}
\abs{\Pr{B}-\Pr{A}}&=\Pr{B}-\Pr{A}  =  \Pr{{B}\setminus {A}}= 
\Pr{B\cap \LL{x < \abs{{\lambda}^{(n)}_0}\leq \fb_n y + \fa_n}^\textnormal{c}}
\\
&\leq
\Pr{\LL{x < \abs{{\lambda}^{(n)}_0}\leq \fb_n y + \fa_n}^\textnormal{c}}\\
&= \Pr{\abs{\lambda^{(n)}_0} \leq x } + \Pr{\abs{\lambda^{(n)}_0} > \fb_n y + \fa_n} \\
& = \Pr{\abs{\frac{\lambda^{(n)}_0}{\sqrt{n}} } \leq \frac{x}{\sqrt{n}} } + \Pr{\abs{\frac{\lambda^{(n)}_0}{\sqrt{n}}} > \frac{\fb_n y + \fa_n}{\sqrt{n}} }.
\end{split}
\end{equation}
By~\eqref{eq:eigenvals} we have
 $\lambda^{(n)}_0=\sum_{j=0}^{n-1}\xi_j$. Since $\xi_0,\ldots,\xi_{n-1}$ are i.i.d. non-degenerate zero mean  random variables with finite second moment,
the Central Limit Theorem yields
\begin{equation}\label{eq:clt1}
\frac{\lambda^{(n)}_0}{\sqrt{n}}\longrightarrow \textnormal{N}(0,\mathbb{E}[|\xi_0|^{2}), \quad \textrm{ in distribution, } \quad \textrm{ as }\quad n\to \infty,
\end{equation}
where $\textnormal{N}(0,\mathbb{E}[|\xi_0|^{2})$ denotes the Gaussian distribution with zero mean and variance $\mathbb{E}[|\xi_0|^{2}]$.
We note that for any $x,y\in \mathbb{R}$ the following limits holds
\begin{equation}\label{eq:limites1}
\lim\limits_{n\to \infty}\frac{x}{\sqrt{n}}=0\quad \textrm{ and } \quad 
\lim\limits_{n\to \infty}\frac{\fb_n y + \fa_n}{\sqrt{n}}=\infty.
\end{equation}
Hence~\eqref{eq:removing1},~\eqref{eq:clt1} 
and~\eqref{eq:limites1} with the help of 
Lemma~\ref{lem:doublelimit} in 
Appendix~\ref{ap:tools} imply~\eqref{eq:limitremoving}.              
\end{proof}
\subsection{\textbf{Reduction to the bounded case}}\label{sec:bounded}
In this subsection, we prove that it is enough to prove Theorem~\ref{thm:generalcase} for bounded random variables. That is,
the random variables $(\xi_j)_{j\in \{0,\ldots,n-1\}}$ in $\lambda^{(n)}_k = \sum_{j=0}^{n-1} \om^{kj}_n \xi_j$ can be considered bounded for all $j\in \{0,\ldots,n-1\}$. 
Following the spirit of Lemma~4 in \cite{BrycSethuraman} we obtain 
the following comparison.
\begin{lem}[Truncation procedure]\label{lem:bounded}
Assume that Hypothesis~\ref{hyp:lc} is valid. 
For each
 $n\in \mathbb{N}$ and $j\in \{0,\ldots n-1\}$ define
the array of random variables $\big(\widetilde{\xi}_j^{(n)}\big)_{j\in\{0,\ldots,n-1\}}$ by
\begin{equation}\label{eq:defarray}
\widetilde{\xi}_j^{(n)}:= \xi_j\Ind{\abs{\xi_j}\leq n^{1/s}} - \E{\xi_j\Ind{\abs{\xi_j}\leq n^{1/s}}},
\end{equation}
where $s=2+\delta$ and $\delta$ is the constant that appears in Hypothesis~\ref{hyp:lc}.
For each $k\in \{1,\ldots,n-1\}$, set
\begin{equation}\label{e:eigentilde}
\widetilde{\lambda}^{(n)}_k:= \sum_{j=0}^{n-1} \widetilde{\xi}^{(n)}_j\, \om^{kj}_n.
\end{equation} 
Then it follows that 
\begin{equation}\label{eq:limite01}
\mathbb{P}\left(\lim\limits_{n\to \infty}
\max_{1\leq k \leq n-1} \abs{\abs{\lambda^{(n)}_k} - \abs{\widetilde{\lambda}^{(n)}_k}}=0\right)=1.
\end{equation}
\end{lem}
\begin{proof}
Let $k\in \{1,\ldots,n-1\}$ be fixed. Recall that $\om_n^{k} = \exp(\mathsf{i} k \cdot2\pi/n)$. Note $w_n^{k} \neq 1$ and $w_n^{kn}= 1$. Hence the geometric sum $\sum_{j=0}^{n-1} \om_n^{kj}=0$. Then 
\begin{equation}\label{eq:sumazero}
\sum_{j=0}^{n-1} \om_n^{kj}  \widetilde{\xi}^{(n)}_j 
=
\sum_{j=0}^{n-1} \om_n^{kj} \xi_j \Ind{\abs{\xi}\leq n^{1/s}}.
\end{equation}
Indeed,
\begin{equation}
\begin{split}
\sum_{j=0}^{n-1} \om_n^{kj}  \widetilde{\xi}^{(n)}_j  & = \sum_{j=0}^{n-1}  \om_n^{kj}\xi_j\Ind{\abs{\xi_j}\leq n^{1/s}} -  \sum_{j=0}^{n-1} \om_n^{kj}\E{\xi_j\Ind{\abs{\xi_j}\leq n^{1/s}}} \\
& =  \sum_{j=0}^{n-1}  \om_n^{kj}\xi_j\Ind{\abs{\xi_j}\leq n^{1/s}} - \E{\xi_0\Ind{\abs{\xi_0}\leq n^{1/s}}} \sum_{j=0}^{n-1} \om_n^{kj} \\
& =  \sum_{j=0}^{n-1}  \om_n^{kj}\xi_j\Ind{\abs{\xi_j}\leq n^{1/s}}.
\end{split}
\end{equation}
As a consequence of~\eqref{eq:sumazero} and the triangle inequality we have the following estimate
\begin{equation}\label{eq:cotasuperior}
\begin{split}
\abs{\abs{\sum_{j=0}^{n-1} \om_n^{kj}  \xi_j } - \abs{\sum_{j=0}^{n-1} \om_n^{kj}  \widetilde{\xi}^{(n)}_j }} & \leq \abs{\sum_{j=0}^{n-1} \om_n^{kj}  \xi_j - \sum_{j=0}^{n-1} \om_n^{kj}  \widetilde{\xi}^{(n)}_j} \\
& = \abs{\sum_{j=0}^{n-1} \om_n^{kj}  \xi_j - \sum_{j=0}^{n-1} \om_n^{kj} \xi_j \Ind{\abs{\xi_j}\leq n^{1/s}}} \\
& = \abs{\sum_{j=0}^{n-1} \om_n^{kj} \xi_j \Ind{\abs{\xi_j}> n^{1/s}}} \leq \sum_{j=0}^{n-1} \abs{\xi_j} \Ind{\abs{\xi_j}> n^{1/s}}.
\end{split}
\end{equation}
Since $\E{\abs{\xi_0}^s}<\infty$, we have $\sum_{\ell=1}^\infty\Pr{\abs{\xi_0}^s > \ell} <\infty$, (see for instance Theorem~4.26 in \cite{Klenke}).
Hence $\sum_{\ell=1}^\infty \Pr{\abs{\xi_0} > \ell^{1/s}} <\infty$. 
By the Borel--Cantelli Lemma (see Theorem~2.7 item (i) in \cite{Klenke}), we have that 
\[
\mathbb{P}\left(\limsup_{\ell\to\infty} \LL{\abs{\xi_\ell} > \ell^{1/s}}\right)=0.
\] In other words, there exists an event $\Omega^*$ with $\Pr{\Omega^*}=1$ such that for each $v\in\Omega^*$ there is $L(v)\in\N$ satisfying 
\begin{equation}\label{eqn:27041521}
\abs{\xi_\ell(v)}\leq \ell^{1/s}\quad\mbox{ for all } \ell\geq L(v). 
\end{equation}
Let $v\in\Omega^*$ and define $n\geq 1+\max\LL{L(v),\abs{\xi_0(v)}^s,\ldots,\abs{\xi_{L(v)}(v)}^s}$. 
By~\eqref{eqn:27041521} and the definition of $n$, we obtain
\begin{equation}\label{eq:cotazero}
\begin{split}
\sum_{j=0}^{n-1} \abs{\xi_j} \Ind{\abs{\xi}> n^{1/s}} & \leq \sum_{j=0}^{L(v)} \abs{\xi_j} \Ind{\abs{\xi_j}> n^{1/s}} +  \sum_{j=L(v)+1}^{n-1} \abs{\xi_j} \Ind{\abs{\xi_j}> n^{1/s}} \\
& \leq \sum_{j=0}^{L(v)} \abs{\xi_j} \Ind{\abs{\xi_j}> n^{1/s}} +  \sum_{j=L(v)+1}^{n-1} \abs{\xi_j} \Ind{\abs{\xi_j}> j^{1/s}}
= 0.
\end{split}
\end{equation}
Combining~\eqref{eq:cotasuperior} and~\eqref{eq:cotazero} 
we deduce~\eqref{eq:limite01}.
\end{proof}
Next lemma allows us to replace the original array of random variables $\PP{{\xi}_j^{(n)}}_{j\in \{0,\ldots,n-1\}}$ by  array $\PP{\widetilde{\xi}_j^{(n)}}_{j\in \{0,\ldots,n-1\}}$ of bounded random variables defined in~\eqref{eq:defarray} of Lemma~\ref{lem:bounded}.
\begin{lem}[Reduction to the bounded case]\label{lem:acotado}
Assume that Hypothesis~\ref{hyp:lc} is valid.
Let  
\begin{equation}\label{eq:tildeminmax}
\widetilde{\sigma}^{(n)}_{\min} := \min_{k\in\LL{1,\ldots,n-1}}{ \abs{\widetilde{\lambda}^{(n)}_k}}\quad \textrm{ and } \quad \widetilde{\sigma}^{(n)}_{\max} := \max_{k\in\LL{1,\ldots,n-1}}{\abs{\widetilde{\lambda}^{(n)}_k}},
\end{equation}
where $(\widetilde{\lambda}^{(n)}_k)_{k\in \{1,\ldots, n-1\}}$ is defined in~\eqref{e:eigentilde} of 
Lemma~\ref{lem:bounded}. 
Then for any $x,y\in \mathbb{R}$ it follows that
\begin{equation}\label{eqn:eqn200520210851}
\lim\limits_{n\to \infty}\left|
\Pr{\sigma^{(n)}_{\min}\leq  x,\frac{\sigma^{(n)}_{\max}-\fa_n}{\fb_n}\leq y} - \Pr{\widetilde{\sigma}^{(n)}_{\min}\leq  x,\frac{\widetilde{\sigma}^{(n)}_{\max}-\fa_n}{\fb_n}\leq y}\right| = 0.
\end{equation}
\end{lem}
\begin{proof}
For each $n\in \mathbb{N}$ let
\[A_n:=\bigcap_{k=1}^{n-1} \LL{x < \abs{\lambda^{(n)}_k}\leq \fb_n y + \fa_n}\quad \textrm{ and }\quad B_n:=\bigcap_{k=1}^{n-1} \LL{x < \abs{\widetilde{\lambda}^{(n)}_k}\leq \fb_n y + \fa_n}.\] 
Lemma~\ref{lem:bounded} implies that almost surely
$\mathds{1}_{A_n} - \mathds{1}_{B_n}\to 0$, as $n\to \infty$.
Hence, the Dominated Convergence Theorem yields
\begin{equation}\label{eqn:reduction030520211656}
\lim\limits_{n\to \infty}\left|\Pr{A_n} - \Pr{B_n}\right|=0.
\end{equation}
The preceding limit with the help of 
Lemma~\ref{lem:DifZero070520291540} and~\eqref{eq:equiv} and implies~\eqref{eqn:eqn200520210851}.
\end{proof}
\subsection{\textbf{Smoothness}}\label{sec:smooth}
In this subsection we introduce a Gaussian perturbation  in order to smooth the random variables 
$(\widetilde{\xi}^{(n)}_j)_{j\in \{1,\ldots,n-1\}}$ defined in~\eqref{eq:defarray} of Lemma~\ref{lem:bounded}.
Let $(N_j)_{j\in \{1,\ldots,n-1\}}$ be an array of random variables with standard Gaussian distribution. Let $(\fs_n)_{n\in \mathbb{N}}$ be a deterministic sequence of positive numbers satisfying
$\fs_n\to 0$ as $n\to \infty$ 
in a way that
\begin{equation}\label{eq:loquedan}
\lim\limits_{n\to \infty}  \frac{\fs_n\fb_n}{\sqrt{n}}=\lim\limits_{n\to \infty}  \frac{\fs_n\fa_n}{\sqrt{n}}=0.
\end{equation}
This is precise in~\eqref{eq:defsn2} below.
We anticipate that $\fs_n\approx n^{-\theta}$ for a suitable positive exponent $\theta$.
We define the array $(\gamma_k)_{k\in \{1,\ldots,n-1\}}$ as follows
\begin{equation}\label{eq:gammadef}
\gamma^{(n)}_k := \sum_{j=0}^{n-1} \om^{kj}_n \fs_n N_j = \fs_n\sum_{j=0}^{n-1} \om^{kj}_n  N_j\quad\textrm{ for }\quad k\in \{1,\ldots,n-1\}.
\end{equation}
Then we set 
\begin{equation}\label{eq:sigmaminmax}
\begin{split}
&\beta_n:=\sqrt{\frac{2}{n}},\quad
\displaystyle\sigma^{(n,\NN)}_{\min} := \beta_n\min_{1\leq k\leq n-1}\LL{ \abs{\widetilde{\lambda}^{(n)}_k + \gamma^{(n)}_k}},\quad
\displaystyle\sigma^{(n,\NN)}_{\max}:= \beta_n\max_{1\leq k\leq n-1}\LL{ \abs{\widetilde{\lambda}^{(n)}_k + \gamma^{(n)}_k}},
\end{split}
\end{equation}
where $(\widetilde{\lambda}^{(n)}_k)_{k\in\{1,\ldots,n-1\}}$ is defined in~\eqref{e:eigentilde} of Lemma~\ref{lem:bounded}.
\begin{lem}\label{lem:lem200520210925}
Assume that Hypothesis~\ref{hyp:lc} is valid.
 Then 
 for any $x,y\in \mathbb{R}$ 
and a suitable sequence $(\fs_n)_{n\in \mathbb{N}}$ that tends to zero as $n\to \infty$,
 it follows that
\begin{equation}\label{eq:limitesmooth}
\lim\limits_{n\to \infty}
\Pr{\widetilde{\sigma}^{(n)}_{\min}\leq x,\;  \widetilde{\sigma}^{(n)}_{\max} \leq \fb_n y + \fa_n} =R(x)G(y),
\end{equation}
where $\widetilde{\sigma}^{(n)}_{\min}$ and $\widetilde{\sigma}^{(n)}_{\max}$ are given 
in~\eqref{eq:tildeminmax} of Lemma~\ref{lem:acotado} and $R,G$ are defined in Theorem~\ref{thm:generalcase}.
\end{lem}
\begin{proof}
Let $k\in \{1,\ldots,n-1\}$. The triangle inequality yields
\begin{equation}\label{eqn:eqn120520211539}
\abs{\widetilde{\lambda}^{(n)}_k + \gamma^{(n)}_k} - \max_{1\leq \ell\leq n-1}{\abs{\gamma^{(n)}_\ell}} \leq \abs{\widetilde{\lambda}^{(n)}_k} \leq \abs{\widetilde{\lambda}^{(n)}_k + \gamma^{(n)}_k} + \max_{1\leq \ell\leq n-1}{\abs{\gamma^{(n)}_\ell}}.
\end{equation} 
By~\eqref{eqn:eqn120520211539} we have
\begin{equation}\label{eqn:eqn120520211624}
\begin{split}
    & \LL{ \sigma^{(n,\NN)}_{\min} + \beta_n \max_{1\leq \ell\leq n-1}\abs{\gamma^{(n)}_\ell} \leq \beta_n x,\; \sigma^{(n,\NN)}_{\max} + \beta_n \max_{1\leq \ell\leq n-1}{\abs{\gamma^{(n)}_\ell}}\leq \beta_n (\fb_n y + \fa_n) } \\
    & \quad \subset \LL{ \sigma^{(n,\NN)}_{\min} + \beta_n \max_{1\leq \ell\leq n-1}\abs{\gamma^{(n)}_\ell} \leq \beta_n x,\; \beta_n \widetilde{\sigma}^{(n)}_{\max} \leq \beta_n (\fb_n y + \fa_n) } \\
    & \quad \subset \LL{\beta_n  \widetilde{\sigma}^{(n)}_{\min}\leq \beta_n x,\;  \beta_n \widetilde{\sigma}^{(n)}_{\max} \leq \beta_n (\fb_n y + \fa_n)} \\
     & \quad = \LL{ \widetilde{\sigma}^{(n)}_{\min}\leq \ x,\;   \widetilde{\sigma}^{(n)}_{\max} \leq  (\fb_n y + \fa_n)} \\
    & \quad \subset \LL{ \sigma^{(n,\NN)}_{\min} - \beta_n \max_{1\leq \ell\leq n-1}\abs{\gamma^{(n)}_\ell} \leq \beta_n x,\; \beta_n \widetilde{\sigma}^{(n)}_{\max} \leq \beta_n (\fb_n y + \fa_n) } \\
    & \quad \subset \LL{ \sigma^{(n,\NN)}_{\min} - \beta_n \max_{1\leq \ell\leq n-1}\abs{\gamma^{(n)}_\ell} \leq \beta_n x,\; \sigma^{(n,\NN)}_{\max} - \beta_n \max_{1\leq \ell\leq n-1}\abs{\gamma^{(n)}_{\ell}}\leq \beta_n (\fb_n y + \fa_n) },
\end{split}
\end{equation}
where 
$\beta_n$,
$\sigma^{(n,\NN)}_{\min}$ and $\sigma^{(n,\NN)}_{\max}$ are defined in~\eqref{eq:sigmaminmax}.
We claim that
\begin{equation}\label{eq:maxzero}
\beta_n
\max_{1\leq \ell\leq n-1}\abs{\gamma^{(n)}_\ell}\longrightarrow 0, \quad \textrm{ in distribution, } \quad \textrm{ as }\quad n\to \infty.
\end{equation}
Indeed, by
Proposition~\ref{thm:Gaussiancase},  Slutsky's Theorem and~\eqref{eq:loquedan}
we have 
\begin{equation}\label{eq:limitZ}
\begin{split}
\beta_n
\max_{1\leq \ell\leq n-1}\abs{\gamma^{(n)}_\ell}
=
\fs_n\beta_n\fb_n\frac{\frac{1}{\fs_n}
\max\limits_{1\leq \ell\leq n-1}\abs{\gamma^{(n)}_\ell}-\fa_n}{\fb_n}+\fs_n\beta_n \fa_n
\longrightarrow 0, 
\end{split}
\end{equation}
in distribution,  as $n\to \infty$.
The limit~\eqref{eq:limitZ} can be also {deduced} from (1.1), p. 522 in \cite{DavisMikosh} or Chapter~10 in \cite{Brockwell}.

Now, we have the necessary elements to conclude the proof of Lemma~\ref{lem:lem200520210925}.
By Lemma~\ref{lem:noexiste} we have
\begin{equation}\label{eq:perturbadas}
\lim\limits_{n\to \infty}
\Pr{ \sigma^{(n,\NN)}_{\min} \leq \beta_n x,\; \sigma^{(n,\NN)}_{\max} \leq \beta_n(\fb_n y + \fa_n)}=R(x)G(y)\quad \textrm{ for any } x\geq 0,~y\in \mathbb{R}.
\end{equation}
By~\eqref{eq:maxzero} and~\eqref{eq:perturbadas} with the help of the  Slutsky Theorem  we deduce 
\begin{equation}
\begin{split}
&\lim\limits_{n\to \infty}
\Pr{ \sigma^{(n,\NN)}_{\min} - \beta_n\max_{1\leq \ell\leq n-1}\abs{\gamma^{(n)}_\ell} \leq \beta_n x,\; \sigma^{(n,\NN)}_{\max} - \beta_n \max_{1\leq \ell\leq n-1}\abs{\gamma^{(n)}_{\ell}}\leq \beta_n(\fb_n y + \fa_n) }\\
&=\lim\limits_{n\to \infty}
\Pr{ \sigma^{(n,\NN)}_{\min} + \beta_n\max_{1\leq \ell\leq n-1}\abs{\gamma^{(n)}_\ell} \leq \beta_n x,\; \sigma^{(n,\NN)}_{\max} + \beta_n \max_{1\leq \ell\leq n-1}\abs{\gamma^{(n)}_{\ell}}\leq \beta_n(\fb_n y + \fa_n) }\\
& =R(x)G(y)\quad \textrm{ for any } x\geq 0,~y\in \mathbb{R}.
\end{split}
\end{equation}
The preceding limit with the help 
of~\eqref{eqn:eqn120520211624} implies
\[
\lim\limits_{n\to \infty}\Pr{\widetilde{\sigma}^{(n)}_{\min}\leq x,\;  \widetilde{\sigma}^{(n)}_{\max} \leq \fb_n y + \fa_n}=R(x)G(y)\quad \textrm{ for any } x\geq 0,~y\in \mathbb{R}.
\]
As a consequence we conclude~\eqref{eq:limitesmooth}.
\end{proof}
\subsection{\textbf{Bounded and smooth case}}\label{sec:boundedsmooth}
In the subsection we prove the following lemma.
\begin{lem}[Gaussian approximation for the bounded and smooth case]\label{lem:noexiste}  
Assume that Hypothesis~\ref{hyp:lc} is valid.
Let $\beta_n$, $\sigma^{(n,\NN)}_{\min}$ and $\sigma^{(n,\NN)}_{\max}$ being as in~\eqref{eq:sigmaminmax}.
 Then 
 for any $x,y\in \mathbb{R}$  and a suitable sequence $(\fs_n)_{n\in \mathbb{N}}$ that tends to zero as $n\to \infty$, 
 it follows that
\begin{equation}\label{eq:perturbadaslem}
\lim\limits_{n\to \infty}
\Pr{ \sigma^{(n,\NN)}_{\min} \leq \beta_n x,\; \sigma^{(n,\NN)}_{\max} \leq \beta_n(\fb_n y + \fa_n)}=R(x)G(y)\quad \textrm{ for any } x\geq 0,~y\in \mathbb{R},
\end{equation}
where $R$ and $G$ are defined in 
Theorem~\ref{thm:generalcase}.
\end{lem}
We follow the ideas from \cite{DavisMikosh}.
To prove Lemma~\ref{lem:noexiste} we introduce some notation and the so-called Einmahl--Koml\'os--Major--Tusn\'ady approximation.
For each $d\in \mathbb{N}$ and indexes $i_j\in \{1,\ldots,n-1\}$, $j=1,\ldots,d$ we 
define the Fourier frequencies  $w_{i_j} = \frac{2\pi i_j}{n}$, $j=1,\ldots,d$ and then the vector 
\begin{equation}\label{eqn:FourierFreq}
v_d(\ell) = \PP{ \cos(w_{i_1}\ell), \sin(w_{i_1}\ell),\ldots, \cos(w_{i_d}\ell), \sin(w_{i_d}\ell) }^{\textnormal{T}}\quad \textrm{ for any }\quad \ell\in \mathbb{N}_0,
\end{equation} 
where 
${\textnormal{T}}$ denotes the transpose operator.
The next lemma is  the main tool in the proof of 
Lemma~\ref{lem:noexiste}. It allows us to reduce the problem to a perturbed Gaussian case. 
\begin{lem}[Lemma~3.4 of \cite{DavisMikosh}]\label{lem:NormalApprox}
Let $d,n\in \mathbb{N}$ and denote by $\widetilde{p}_n$ the continuous density function of  the random vector
\[
2^{1/2}n^{-1/2} \sum_{\ell=0}^{n-1}\PP{\widetilde{\xi}_{\ell}^{(n)} + s_n N_\ell} v_d(\ell),
\] where $\PP{N_\ell}_{t\in \mathbb{N}_0}$ is a sequence of i.i.d. standard Gaussian  random variables, independent of the sequence 
$\PP{\widetilde{\xi}_{\ell}^{(n)}}_{\ell \in \{0,\ldots,n-1\}}$ that is defined in~\eqref{eq:defarray} of Lemma~\ref{lem:bounded}, and $s_n^2 = \textnormal{Var}\PP{\widetilde{\xi}_0^{(n)}}q_n^2$. If 
\[
n^{-2c}\ln(n) \leq q_n^2 \leq 1\quad \textrm{ with }\quad c=\frac{1}{2}-\frac{1-\eta}{2+\delta}
\] for arbitrarily small $\eta>0$, then the relation
\[
\widetilde{p}_n(x) = \varphi_{(1+s^2_n)I_{2d}}(x)(1+\textnormal{o}_n(1))
\] holds uniformly for $\abs{x}^3=\textnormal{o}^{(d)}_n\PP{\min\LL{n^c,n^{1/2-1/(2+\delta)}}}$, where the implicit constant in the $\textnormal{o}^{(d)}_n$-notation depends on $d$, and $\varphi_\Sigma$ is the density of a $2d$-dimensional zero mean Gaussian vector with covariance matrix $\Sigma$.
\end{lem}
\begin{proof}[Proof of Lemma~\ref{lem:noexiste}]
Let $x,y\in \mathbb{R}$.
The idea is to apply 
Lemma~\ref{lem:NormalApprox} to the random sequence $\left(\widetilde{\xi}_j^{(n)}+\fs_n N_j\right)_{j\in \mathbb{N}_0}$ with a suitable deterministic sequence $(\fs_n)_{n\in \mathbb{N}}$.
We note that the variance of $\widetilde{\xi}_0^{(n)}$,
$\textnormal{Var}\PP{\widetilde{\xi}_0^{(n)}}$, is bounded by $\mathbb{E}[\xi^2_0]=1$.
Then we define 
\begin{equation}\label{eq:defsn2}
\fs^2_n:= \textnormal{Var}\PP{\widetilde{\xi}_0^{(n)}} n^{-1/2+(1-\eta)/(2+\delta)}
\end{equation}
{for} sufficiently small $\eta>0$.
Since
\begin{equation}\label{eqn:eqn251020211045}
\begin{split}
    & \Pr{ \sigma^{(n,\NN)}_{\min} \leq \beta_n x,\; \sigma^{(n,\NN)}_{\max}\leq \beta_n(\fb_n y + \fa_n)} \\
    &  = \Pr{ \bigcap_{k=1}^{n-1}\LL{\abs{\widetilde{\lambda}^{(n)}_k + \gamma^{(n)}_k}\leq \fb_n y + \fa_n}} - \Pr{\bigcap_{k=1}^{n-1}\LL{x < \abs{\widetilde{\lambda}^{(n)}_k + \gamma^{(n)}_k}\leq \fb_n y + \fa_n}}  \\
        & = \Pr{ \bigcap_{k=1}^{q_n}\LL{\abs{\widetilde{\lambda}^{(n)}_k + \gamma^{(n)}_k}\leq \fb_n y + \fa_n}} - \Pr{\bigcap_{k=1}^{q_n}\LL{x < \abs{\widetilde{\lambda}^{(n)}_k + \gamma^{(n)}_k}\leq \fb_n y + \fa_n}},
\end{split}
\end{equation}
where $q_n:=\lfloor n/2\rfloor$ with $\lfloor \cdot \rfloor$ being the floor function.
Then we define 
\begin{equation}\label{eq:limit2}
J^{(n)}(x,y):=\Pr{\bigcap_{k=1}^{q_n}\LL{x < \abs{\widetilde{\lambda}^{(n)}_k + \gamma^{(n)}_k}\leq \fb_n y + \fa_n}}.
\end{equation}
In what follows we compute the limit of~\eqref{eq:limit2} as $n\to\infty$. 
In fact we prove that
\begin{equation}\label{eq:limit3}
\lim\limits_{n\to \infty}
J^{(n)}(x,y)=\exp\left(-\PP{\frac{x^2}{2} + e^{-y}}\right).
\end{equation}
For convenience, let
\begin{equation}\label{eq:indenev}
A^{(n)}_{k}:=\LL{x < \abs{\widetilde{\lambda}^{(n)}_k + \gamma^{(n)}_k}\leq \fb_n y + \fa_n}^{\textnormal{c}}\quad \textrm{ for all }\quad k\in \{1,\ldots,q_n\}
\end{equation}
and observe that
\begin{equation}\label{eq:unomenos}
1-J^{(n)}_2=1 - \Pr{\bigcap_{k=1}^{q_n}
\left(A^{(n)}_{k}\right)^{\textnormal{c}}
} =\Pr{\bigcup_{k=1}^{q_n}A^{(n)}_{k}}.
\end{equation}
By Lemma~\ref{lem:bonferroni} in Appendix~\ref{ap:tools} for any fixed $\ell\in \mathbb{N}$ we obtain
\begin{equation}\label{eqn:Bonferroniine}
    \sum_{j=1}^{2\ell}(-1)^{j-1} S^{(n)}_j \leq \Pr{\bigcup_{k=1}^{q_n} A^{(n)}_{k}} \leq \sum_{j=1}^{2\ell-1} (-1)^{j-1} S^{(n)}_j,
\end{equation}
where 
\begin{equation}\label{eq:bonfe}
S^{(n)}_j = \sum_{1\leq i_1 <\cdots <i_j\leq q_n} \mathbb{P}\PP{A^{(n)}_{i_1}\cap\cdots\cap A^{(n)}_{i_j}}.
\end{equation}
We claim that for every fixed $d\in \mathbb{N}$ the following limit holds true
\begin{equation}\label{eq:qndlimit}
\lim\limits_{n\to \infty}
\binom{q_n}{d}\Pr{\bigcap_{k=1}^d A^{(n)}_{k} }=  \frac{1}{d!}\PP{\frac{x^2}{2} + e^{-y}}^d,
\end{equation}
where the symbol $\binom{q_n}{d}$ denotes the binomial coefficient.
Indeed, by Lemma~\ref{lem:NormalApprox} we have
\begin{equation}\label{eq:KMTapprox}
\begin{split}
\Pr{\bigcap_{k=1}^d A^{(n)}_{k} }&= \Pr{\bigcap_{k=1}^d \LL{\frac{2x^2}{n} < \abs{\sqrt{\frac{2}{n}}\PP{\widetilde{\lambda}^{(n)}_{i_k} + \gamma^{(n)}_{i_k}}}^2 \leq \frac{2\PP{\fb_n y + \fa_n}^2}{n}}^\textnormal{c}}. \\
    & = (1+\textnormal{o}_n(1))\int_{{B}^{(n)}_d} \varphi_{(1+\fs_n^2)I_{2d}}(u) \mathrm{d}u,
\end{split}
\end{equation}
where  $I_{2d}$ denotes the $2d\times 2d$ identity matrix, 
$\varphi_{(1+\fs_n^2)I_{2d}}$ is the density of a $2d$-dimensional Gaussian vector with zero mean and covariance matrix $(1+\fs_n^2)I_{2d}$, 
$\textnormal{o}_n(1)\to 0$ as $n\to \infty$,
and
\begin{equation}
\begin{split}
{B}^{(n)}_d & :=\left\{\PP{w_1,v_1,\ldots,w_d,v_d} \in \mathbb{R}^{2d}: w_i^2+v_i^2\leq \frac{2x^2}{n}\quad \textrm{ or } \right.\\
&\left. \hspace{5.9cm}
w_i^2+v_i^2 > 2\frac{\PP{\fb_n y + \fa_n}^2}{n}\quad \textrm{ for all }\quad i\in\{1,\ldots,d\}\right\}.
\end{split}
\end{equation} 
By~\eqref{eq:defab} and since $x,y\in \mathbb{R}$ are fixed,
there exists $n_0=n_0(x,y)$ such that 
\[2x^2/n<2{\PP{\fb_n y + \fa_n}^2}/n\quad \textrm{ for all }\quad n\geq n_0.\] 
Hence,~\eqref{eq:KMTapprox} with the help of 
Lemma~\ref{lem:2Ray} in Appendix~\ref{ap:tools} yields
\begin{equation}\label{eq:formula1o}
\Pr{\bigcap_{k=1}^d A^{(n)}_{k} }=  (1+\textnormal{o}_n(1))\PP{1-\exp\PP{-\frac{x^2}{n(1+\fs^2_n)}} + \exp\PP{-\frac{(\fb_n y + \fa_n)^2}{n(1+\fs^2_n)}}}^d.
\end{equation}
By~\eqref{eq:defab}, the Stirling formula (see formula~(1) in \cite{Robbins1955}) and the fact that $\fs_n\to 0$ as $n\to \infty$ we deduce 
\begin{equation}\label{eq:formula1o1}
\lim\limits_{n\to \infty}
\binom{q_n}{d}\PP{1-\exp\PP{-\frac{x^2}{n(1+\fs^2_n)}} + \exp\PP{-\frac{(\fb_n y + \fa_n)^2}{n(1+\fs^2_n)}}}^d = \frac{1}{d!}\PP{\frac{x^2}{2} + e^{-y}}^d.
\end{equation}
As a consequence of~\eqref{eq:formula1o} 
and~\eqref{eq:formula1o1} we obtain~\eqref{eq:qndlimit}.

Now, we prove~\eqref{eq:limit3}.
By~\eqref{eqn:Bonferroniine},~\eqref{eq:bonfe},
~\eqref{eq:qndlimit} 
and the fact that the events $(A^{(n)}_{k})_{k\in \{1,\ldots,q_n\}}$ are independent,
we have for any $\ell\in \mathbb{N}$ 
\begin{equation}
\begin{split}
    \sum_{j=1}^{2\ell}(-1)^{j-1} \frac{1}{j!}\PP{\frac{x^2}{2} + e^{-y}}^j &\leq \liminf\limits_{n\to \infty}\Pr{\bigcup_{k=1}^{q_n} A^{(n)}_{k}}  \\
 &\leq
\limsup\limits_{n\to \infty}\Pr{\bigcup_{k=1}^{q_n} A^{(n)}_{k}}\leq     
     \sum_{j=1}^{2\ell-1} (-1)^{j-1} \frac{1}{j!}\PP{\frac{x^2}{2} + e^{-y}}^j.
\end{split}
\end{equation}
Sending $\ell\to \infty$ in the preceding inequality yields
\begin{equation}
\begin{split}
\lim\limits_{n\to \infty}\Pr{\bigcup_{k=1}^{q_n} A^{(n)}_{k}}=     
     \sum_{j=1}^{\infty} (-1)^{j-1} \frac{1}{j!}\PP{\frac{x^2}{2} + e^{-y}}^j=1-\exp\left(-\PP{\frac{x^2}{2} + e^{-y}}\right).
\end{split}
\end{equation}
The preceding limit with the help of~\eqref{eq:unomenos} implies~\eqref{eq:limit3}.

Finally, by~\eqref{eqn:eqn251020211045},~\eqref{eq:limit2} and~\eqref{eq:limit3} we obtain
\begin{equation*}
\begin{split}
\lim\limits_{n\to \infty}
\Pr{ \sigma^{(n,\NN)}_{\min} \leq \beta_n x,\; \sigma^{(n,\NN)}_{\max}\leq \beta_n(\fb_n y + \fa_n)}
&=\lim\limits_{n\to \infty}\left(J^{(n)}(0,y)-J^{(n)}(x,y)\right)\\
&=\exp\left(-e^{-y}\right)-\exp\left(-\PP{\frac{x^2}{2} + e^{-y}}\right)\\
&=R(x)G(y),
\end{split}
\end{equation*}
where $R$ and $G$ are defined in~\eqref{eq:defRG}.
\end{proof}
\subsection{\textbf{Proof of Theorem~\ref{thm:generalcase}}}\label{sec:prueba}
In the subsection, we stress the fact that 
Theorem~\ref{thm:generalcase} is just a consequence of what
we have already stated up to here.
\begin{proof}[Proof of Theorem~\ref{thm:generalcase}]
Let $x,y\in \mathbb{R}$.
By~\eqref{eq:limitesmooth} in 
Lemma~\ref{lem:lem200520210925}  we have 
\begin{equation}
\lim\limits_{n\to \infty}
\Pr{\widetilde{\sigma}^{(n)}_{\min}\leq x,\;  \widetilde{\sigma}^{(n)}_{\max} \leq \fb_n y + \fa_n}=R(x)G(y),
\end{equation}
where $\widetilde{\sigma}^{(n)}_{\min}$ and $\widetilde{\sigma}^{(n)}_{\max}$ are given 
in~\eqref{eq:tildeminmax} of Lemma~\ref{lem:acotado}.
The preceding limit with the help 
of~\eqref{eqn:eqn200520210851} in Lemma~\ref{lem:acotado} implies
\begin{equation}
\lim\limits_{n\to \infty}
\Pr{\sigma^{(n)}_{\min}\leq  x,\frac{\sigma^{(n)}_{\max}-\fa_n}{\fb_n}\leq y}=R(x)G(y).
\end{equation}
This concludes the proof of Theorem~\ref{thm:generalcase}.
\end{proof}
\appendix
\section{\textbf{Tools}}\label{ap:tools}
The following section contains useful tools that help us to make this article
more fluid. 
The following elementary lemma is crucial in the proof of Proposition~\ref{thm:Gaussiancase}. 
\begin{lem}\label{lem:080720201357}
Let $n\geq 2$ be a given integer number and
let $Y_1,\ldots,Y_n$ be independent random variables. Define $m^{(n)}:=\min\LL{Y_1,\ldots,Y_n}$ and $M^{(n)}:=\max\LL{Y_1,\ldots,Y_n}$. Then 
\[\Pr{m^{(n)}\leq s,M^{(n)}\leq t} = \prod_{j=1}^n \Pr{Y_j\leq t} - \prod_{j=1}^n \Pr{s<Y_j\leq t}
\quad \textrm{ for all }\quad s,t\in \R.
\]
\end{lem}
Since the proof of Lemma~\ref{lem:080720201357} is straightforward, we omit it.

\begin{lem}[Bonferroni's inequality]\label{lem:bonferroni}
Let $\PP{\Omega,\mathcal{F},\mathbb{P}}$ be a probability space.
Let $A_1,\ldots,A_n$ be events. 
Then for every  $\ell\in \mathbb{N}$,
\begin{equation}\label{eqn:Bonferroni}
    \sum_{j=1}^{2\ell}(-1)^{j-1} S_j \leq \mathbb{P}\PP{A_1\cup \cdots\cup A_n} \leq \sum_{j=1}^{2\ell-1} (-1)^{j-1} S_j,
\end{equation}
where 
\[
S_j= \sum_{1\leq i_1 <\cdots <i_j\leq n} \mathbb{P}\PP{A_{i_1}\cap\cdots\cap A_{i_j}}.
\]
\end{lem}
The proof of Lemma~\ref{lem:bonferroni} is given in Section~1.1 ``Inclusion-exclusion Formula" of \cite{LinBai}.
\begin{lem}[Continuity]\label{lem:doublelimit}
Let $(\Omega,\mathcal{F}, \mathbb{P})$ be a probability space.
Let $(X_n)_{n\in \mathbb{N}}$ be a sequence of random variables defined on $\Omega$ and taking values in $\mathbb{R}$.
Assume that $X_n$ converges in distribution to a random variable $X$,  as 
$n\to \infty$. Let $x$ be a continuity point of the distribution function $F_X$ of the random variable $X$  and let
$(a_n(x))_{n\in \mathbb{N}}$ be a deterministic sequence of real numbers such that $a_n(x)\to x$ as $n\to \infty$. Then
\begin{equation}\label{eq:level0}
\lim\limits_{n\to \infty}\mathbb{P}(X_n\leq a_n(x))= F_X(x).
\end{equation}
In addition, if $F_X$ is a continuous function then
\begin{enumerate}
\item[(i)] for any deterministic sequence $(a_n)_{n\in \mathbb{N}}$ such that $a_n\to 0$ as $n\to \infty$ it follows that
\begin{equation}
\lim\limits_{n\to \infty}\mathbb{P}(|X_n|\leq a_n)= 0.
\end{equation}
\item[(ii)]  for any deterministic sequence $(a_n)_{n\in \mathbb{N}}$ satisfying $a_n\to \infty$ as $n\to \infty$ it follows that
\begin{equation}
\lim\limits_{n\to \infty}\mathbb{P}(|X_n|> a_n)= 0.
\end{equation}
\end{enumerate}
\end{lem}
\begin{proof}
We start with the proof of~\eqref{eq:level0}.
Let $x$ be a continuity point of $F_X$ and take $\epsilon>0$.
Then there exists $n_\epsilon:=n_\epsilon(x)$ such that $x-\epsilon<a_n(x)<x+\epsilon$ for all $n\geq n_\epsilon$.
By monotonicity we have  
\begin{equation} 
\begin{split}
\liminf\limits_{n\to \infty}
\mathbb{P}\left(X_n< x-\epsilon\right)&\leq 
\liminf\limits_{n\to \infty}
\mathbb{P}\left(X_n\leq a_n(x)\right)\\
&
\leq 
\limsup\limits_{n\to \infty}
\mathbb{P}\left(X_n\leq a_n(x)\right)
\leq 
\limsup\limits_{n\to \infty}
\mathbb{P}\left(X_n\leq  x+\epsilon\right).
\end{split}
\end{equation}
Hence, the Portmanteau Theorem (see Theorem~13.16 in \cite{Klenke}) implies for any $\epsilon>0$
\begin{equation} 
\begin{split}
F_X(x-\epsilon)\leq 
\liminf\limits_{n\to \infty}
\mathbb{P}\left(X_n\leq a_n(x)\right)
\leq 
\limsup\limits_{n\to \infty}
\mathbb{P}\left(X_n\leq a_n(x)\right)
\leq 
F_X(x+\epsilon).
\end{split}
\end{equation}
Since $x$ is a continuity point of $F_X$, sending $\epsilon\to 0$ we deduce~\eqref{eq:level0}. 

We continue with the proof of item (i) and item (ii).
By the Continuous Mapping Theorem we have 
$|X_n|\to |X|$ in distribution, as $n\to \infty$. 
On the one hand,~\eqref{eq:level0} yields
\begin{equation}
\lim\limits_{n\to \infty}\mathbb{P}(|X_n|\leq a_n)=\mathbb{P}(|X|\leq 0)=\mathbb{P}(|X|=0)=0,
\end{equation}
which finishes the proof of item (i).

On the other hand, let $m>0$ be arbitrary. Then there exists $n_{m}\in \mathbb{N}$ such that $a_n>m$ for all $n\geq n_m$. Hence,
\begin{equation}
\limsup\limits_{n\to \infty}\mathbb{P}(|X_n|>a_n)=\limsup\limits_{n\to \infty}\mathbb{P}(|X_n|> m)=\mathbb{P}(|X|> m),
\end{equation}
which implies item (ii) sending $m\uparrow \infty$.
\end{proof}
\begin{lem}[Fr\'echet distribution as  the inverse of a scaled Rayleigh distribution]\label{lem:1Ray}
Let $X$ be a real random variable with 
the Rayleigh distribution  
$R(x)=(1-\exp\PP{-x^2/2})\Ind{x\geq 0}$.
Let $Y$ be a real random variable with the Fr\'echet distribution 
$F(y)=\exp(-y^{-2})\Ind{y>0}$. Then
$
\sqrt{2}/X\stackrel{d}=Y
$, where the symbol $\stackrel{d}=$ denotes equality in distribution.
\end{lem}
\begin{proof}
Let $y>0$ and note that
\begin{equation}
\mathbb{P}\left(\sqrt{2}/X\leq y\right)=
\mathbb{P}\left(X\geq \sqrt{2}/y\right)=
\exp\PP{-y^{-2}}=F(y).
\end{equation}
The preceding equality concludes the statement.
\end{proof}
\begin{lem}[Exponential distribution as a sum of square independent Gaussian distributions]\label{lem:2Ray}
Let $X_1$ and $X_2$ be two independent random variables with Gaussian distribution with zero mean and variance $\sigma^2>0$.
Then
$X^2_1+X^2_2$ has Exponential distribution with parameter $1/(2\sigma^2)$, that is, 
\[
\mathbb{P}\left(X^2_1+X^2_2\leq x\right)=\left(1-\exp\left(-\frac{x}{2\sigma^2}\right)\right)\Ind{x\geq 0}.
\]
\end{lem}
\begin{proof}
Since the random variables $X_j/\sigma$, $j=1,2$ have standard Gaussian distributions, the random variables $X^2_j/\sigma^2$, $j=1,2$ posses  Chi-square distribution with one degree of freedom.
Due to the independence assumption, we have that $X^2_1/\sigma^2+X^2_2/\sigma^2$ has Chi-square distribution with two degrees of freedom. That is, for any $x\geq 0$, we have 
$
\mathbb{P}\left(X^2_1/\sigma^2+X^2_2/\sigma^2\leq x\right)=1-e^{-x/2}$.
The preceding relation yields
$
\mathbb{P}\left(X^2_1+X^2_2\leq x\right)=1-\exp{\left(-\frac{x}{2\sigma^2}\right)}$
for any $x\geq 0$.
\end{proof}

\section*{\textbf{Acknowledgments}}
G. Barrera would like to express
his gratitude to University of Helsinki, Department of Mathematics and Statistics, for all the facilities used along the realization of this work.
{The authors are grateful to the reviewers for the thorough examination of the paper, which has lead to a significant improvement.}

\section*{Declarations}
\subsection*{Funding} 
The research of G. Barrera has been supported by the Academy of Finland, via 
the Matter and Materials Profi4 University Profiling Action, 
an Academy project (project No. 339228)
and the Finnish Centre of Excellence in Randomness and STructures (project No. 346306).
\subsection*{Availability of data and material}
Data sharing not applicable to this article as no datasets were generated or analyzed during the current study.
\subsection*{Conflict of interests} The authors declare that they have no conflict of interest.
\subsection*{Authors' contributions}
Both authors have contributed equally to the paper.

\bibliographystyle{amsplain}

\begin{thebibliography}{40}
\addcontentsline{}{}{}
\bibitem{aldrovandi} Aldrovandi, R.: \textit{Special matrices of mathematical physics: stochastic, circulant and Bell matrices}.
World Scientific Publishing, (2001).
\url{https://doi.org/10.1142/4772}

\bibitem{AndersonWells} Anderson, W. \& Wells, M.: The exact distribution of the condition number of a Gaussian matrix. \textit{SIAM J. Matrix Anal. Appl.} \textbf{31}, no. 3,  (2009), 1125-1130.
\url{https://doi.org/10.1137/070698932}

\bibitem{ArenasAbreu} Arenas-Velilla, S. \& P\'erez-Abreu, V.: Extremal laws for Laplacian random matrices.
ArXiv:2101.08318.

\bibitem{Ash} Ash, R.: 
\textit{Probability and measure theory}.
Second edition. With contributions by C. Dol\'eans--Dade. Harcourt-Academic Press, Burlington, MA., (2000).
{
\bibitem{Aunger2009}
Auffinger, A., Ben Arous, G. \& P\'ech\'e, S.: 
Poisson convergence for the largest eigenvalues of heavy tailed random matrices.
\textit{Ann. Inst. Henri Poincar\'e Probab. Stat.} \textbf{45}, no. 3, (2009), 589-610.
\url{https://doi.org/10.1214/08-AIHP188}
}
{
\bibitem{Wschebor2004}
Aza\"{i}s, J. \& Wschebor, M.: 
Upper and lower bounds for the tails of the distribution of the condition number of a Gaussian matrix.
\textit{SIAM J. Matrix Anal. Appl.}
\textbf{26}, no. 2, (2004), 426-440.
\url{https://doi.org/10.1137/S0895479803429764}
}

{\bibitem{BaiSilver988}
Bai, Z., Silverstein, J. \& Yin, Y.: A note on the largest eigenvalue of a large-dimensional sample covariance matrix. 
\textit{J. Multivar. Anal.} \textbf{26}, no. 2, (1988), 166-168.
\url{https://doi.org/10.1016/0047-259X(88)90078-4}
\bibitem{BaiYin993}
Bai, Z. \& Yin, Y.:  Limit of the smallest eigenvalue of a large dimensional sample covariance matrix.
\textit{Ann. Probab.} \textbf{21}, no. 3, (1993), 
1275-1294.
\url{https://doi.org/10.1214/aop/1176989118}
}
\bibitem{Barrera2016} Barrera D. \& Peligrad, M.: Quenched limit theorems for Fourier
transforms and periodogram. \textit{Bernoulli} \textbf{22}, no. 1, (2016), 275-301.
\url{https://doi.org/10.3150/14-BEJ658}
{
\bibitem{Bojan2021}
Basrak, B., Cho, Y., Heiny, J. \& Jung, P.:
Extreme eigenvalue statistics of $m$-dependent
heavy-tailed matrices. 
\textit{Ann. Inst. Henri Poincar\'e Probab. Stat.} \textbf{57}, no. 4,  (2021), 2100-2127.
\url{https://doi.org/10.1214/21-AIHP1152}
}
{
\bibitem{BoseGuha2011}
Bose, A., Guha, S., Hazra, R. \& Saha, K.:
Circulant type matrices with heavy tailed entries. 
\textit{Statist. Probab. Lett.} \textbf{81}, no. 11,  (2011), 1706-1716.
\url{https://doi.org/10.1016/j.spl.2011.07.001}
}
{
\bibitem{BoseHachem2020}
Bose, A., Hachem, W.:
Smallest singular value and limit eigenvalue distribution of a class of non-Hermitian random matrices with statistical application. 
\textit{J. Multivariate Anal.} \textbf{178}, (2020), 104623, 24 pp.
\url{https://doi.org/10.1016/j.jmva.2020.104623}
}
{
\bibitem{BoseHazra2012} 
Bose, A.,  Hazra, R. \&  Saha, K.: Extremum of circulant type matrices: a survey. 
\textit{J. Indian Statist. Assoc.} \textbf{50}, no. 1-2, (2012), 21-49.
\url{http://repository.ias.ac.in/121167/}
}
{
\bibitem{BoseHazra2011}
Bose, A., Hazra, R. \&  Saha, K.: 
Poisson convergence of eigenvalues of circulant type matrices. \textit{Extremes} \textbf{14}, no. 4, (2011), 365-392.
\url{https://doi.org/10.1007/s10687-010-0115-5}
}
{
\bibitem{BoseHazra22012}
Bose, A., Hazra, R. \& Saha, K.: 
Product of exponentials and spectral radius of random $k$-circulants. 
\textit{Ann. Inst. Henri Poincar\'e Probab. Stat.} \textbf{48}, no. 2, (2012), 424-443.
\url{https://doi.org/10.1214/10-AIHP404}
}
{
\bibitem{BoseHazra12011}
Bose, A., Hazra, R. \&  Saha, K.: 
Spectral norm of circulant-type matrices. \textit{J. Theoret. Probab.} \textbf{24}, no. 2, (2011), 479-516.
\url{https://doi.org/10.1007/s10959-009-0257-z}
}
{
\bibitem{BoseHazra2010}
Bose, A., Hazra, R. \&  Saha, K.: 
Spectral norm of circulant type matrices with heavy tailed entries. 
\textit{Electron. Commun. Probab.}  \textbf{15}, (2010), 299-313.
\url{https://doi.org/10.1214/ECP.v15-1554}
}
{
\bibitem{BoseMaurya2020}
Bose, A., Maurya, S. \&  Saha, K.:
Process convergence of fluctuations of linear eigenvalue statistics of random circulant matrices.
\textit{Random Matrices Theory Appl.} (2020), 2150032.
\url{https://doi.org/10.1142/S2010326321500325}
}
{
\bibitem{BoseMitra2002}
Bose, A. \& Mitra, J.:
Limiting spectral distribution of a special circulant. 
\textit{Statist. Probab. Lett.} \textbf{60}, no. 1, (2002), 111-120.
\url{https://doi.org/10.1016/S0167-7152(02)00289-4}
}
{
\bibitem{BoseMitra20112}
Bose, A., Mitra, J. \& Sen, A.:
Limiting spectral distribution of random $k$-circulants.
\textit{J. Theoret. Probab.} \textbf{25}, no. 3, (2012), 771-797.
\url{https://doi.org/10.1007/s10959-010-0312-9}
}
\bibitem{arup2018} Bose, A. \& Saha, K.: \textit{Random circulant matrices}. Chapman and Hall/CRC Press, (2018).
\url{https://doi.org/10.1201/9780429435508}
{
\bibitem{BoseSubhraSaha2009}
Bose, A., Subhra H. \&  Saha, K.:
Limiting spectral distribution of circulant type matrices with dependent inputs.
\textit{Electron. J. Probab.} \textbf{14},  no. 86,  (2009), 2463-2491.
\url{https://doi.org/10.1214/EJP.v14-714}
}
\bibitem{Brockwell} Brockwell, P. \&  Davis, R.: \textit{Time series: theory and methods}.
Springer-Verlag New York, (1991).
\url{https://doi.org/10.1007/978-1-4419-0320-4}

\bibitem{BrycSethuraman} Bryc, W. \& Sethuraman, S.: A remark on the maximum eigenvalue for circulant matrices. \textit{Inst. Math. Stat. (IMS) Collect.} \textbf{5}, (2009), 179-184.
\url{https://doi.org/10.1214/09-IMSCOLL512}

\bibitem{burgisser2013condition} B\"{u}rgisser, P. \& Cucker, F.: \textit{Condition: the geometry of numerical algorithms}. Springer-Verlag Berlin Heidelberg, (2013).
\url{https://doi.org/10.1007/978-3-642-38896-5}

{\bibitem{Siegfried2020}
Cerovecki, C., Characiejus, V. \& H\"ormann, S.:
The maximum of the periodogram of a sequence of functional data.
\textit{J. Amer. Statist. Assoc.} (2022).
\url{https://doi.org/10.1080/01621459.2022.2071720}
}

{
\bibitem{Cerovecki2017}
Cerovecki, C. \& H\"ormann, S.:
On the CLT for discrete Fourier transforms of functional time series. 
\textit{J. Multivariate Anal.} \textbf{154},  (2017), 282-295.
\url{https://doi.org/10.1016/j.jmva.2016.11.006}
}
\bibitem{ChenDongarra} Chen, Z. \& Dongarra, J.: Condition numbers of Gaussian random matrices. \textit{SIAM J. Matrix Anal. Appl.} \textbf{27}, no. 3,  (2005), 603-620.
\url{https://doi.org/10.1137/040616413}

\bibitem{Cook2021} Cook, N. \& Nguyen, H.: Universality of the minimum modulus for random trigonometric polynomials. 
{
\textit{Discrete Anal.} \textbf{20}, (2021), 46 pp. \url{https://doi.org/10.19086/da.28985}}

\bibitem{cucker2016probabilistic} Cucker, F.: Probabilistic analyses of condition numbers. \textit{Acta Numer.} \textbf{25}, (2016), 321-382.
\url{https://doi.org/10.1017/S0962492916000027}

\bibitem{David2012} Davis, P.: \textit{ Circulant matrices}. Amer. Math. Soc., (1994).

\bibitem{DavisMikosh} Davis, R. \& Mikosch, T.: The maximum of the periodogram of a non-Gaussian sequence. \textit{Ann. Probab.} \textbf{27}, no. 1, (1999), 522-536.
\url{https://doi.org/10.1214/aop/1022677270}

{
\bibitem{MikoshXie2016}
Davis, R., {Heiny, J.}, Mikosch, T. \& Xie, X.:
Extreme value analysis for the sample autocovariance matrices of heavy-tailed multivariate time series.  
\textit{Extremes} \textbf{19}, no. 3, (2016), 517-547.
\url{https://doi.org/10.1007/s10687-016-0251-7}
}

\bibitem{DemmelJ} Demmel, J.: The probability that a numerical analysis problem is difficult. \textit{Math. Comput.} \textbf{50}, no. 182, (1988), 449-480.
\url{https://doi.org/10.2307/2008617}

\bibitem{edelmanphd} Edelman, A.: \textit{Eigenvalues and condition numbers of random matrices}.
Ph.D. dissertation and numerical analysis report 89-7, MIT, Cambridge, MA, (1989).
Retrieved 2020.11.11, 21.05 h. EET zone
\url{http://math.mit.edu/~edelman/publications/eigenvalues_and_condition_numbers.pdf}

\bibitem{edelman1988} Edelman, A.: Eigenvalues and condition numbers of random matrices. \textit{SIAM J. Matrix Anal. Appl.} \textbf{9}, no. 4, (1988), 543-560.
\url{https://doi.org/10.1137/0609045}

{
\bibitem{Edel1992}
Edelman, A.: On the distribution of a scaled condition number.
\textit{Math. Comp.} \textbf{58}, no. 197, (1992), 
185-190.
\url{https://doi.org/10.1090/S0025-5718-1992-1106966-2}
}
{
\bibitem{Sutton2005}
Edelman, A. \& Sutton, B.:
Tails of condition number distributions. 
\textit{SIAM J. Matrix Anal. Appl.}
\textbf{27}, no. 2, (2005), 547-560.
\url{https://doi.org/10.1137/040614256}
}
\bibitem{Einmahl} Einmahl, U.: Extensions of results of Koml\'os, Major, and Tusn\'ady to the multivariate case. \textit{J. Multivariate Anal.} \textbf{28}, no. 1, (1989), 20-68.
\url{https://doi.org/10.1016/0047-259X(89)90097-3}

\bibitem{GA} Galambos, J.: \textit{The asymptotic theory of extreme order statistics}. Second Edition.
Krieger Publishing Company, Florida, (1987).

\bibitem{vonNG2} Goldstine, H. \& von Neumann, J.: Numerical inverting matrices of high
order II. \textit{Proc. Amer. Math. Soc.} \textbf{2}, no. 2, (1951), 188-202.
\url{https://doi.org/10.2307/2032484}

\bibitem{gray2006} Gray, R.:  Toeplitz and circulant matrices: A review.
\textit{Found. Trends Commun. Inf. Theory} \textbf{2}, no. 3, (2006), 155-239.
\url{https://doi.org/10.1561/0100000006}
{
\bibitem{Gregoratti2021}
Gregoratti, G. \& Maran, D.: 
Least singular value and condition number of a square random matrix with i.i.d. rows.
\textit{Statist. Probab. Lett.} \textbf{173}, no. 109070,  (2021), 7 pp. 
\url{https://doi.org/10.1016/j.spl.2021.109070}
\bibitem{HeinyMikosch2018}
Heiny, J. \& Mikosch, T.:  
Almost sure convergence of the largest and smallest eigenvalues of high-dimensional sample correlation matrices.
\textit{Stochastic Process. Appl.} \textbf{128}, no. 8, (2018), 2779-2815.
\url{https://doi.org/10.1016/j.spa.2017.10.002}}
{
\bibitem{Heiny2017}
Heiny, J. \& Mikosch, T.: 
Eigenvalues and eigenvectors of heavy-tailed sample covariance matrices with general growth rates: the iid case. 
\textit{Stochastic Process. Appl.} \textbf{127}, no. 7, (2017), 2179-2207.
\url{https://doi.org/10.1016/j.spa.2016.10.006}
}
{
\bibitem{Heiny2021}
Heiny, J. \& Mikosch, T.:  
Large sample autocovariance matrices of linear
processes with heavy tails. 
\textit{Stochastic Process. Appl.} \textbf{141}, (2021), 344-375.
\url{https://doi.org/10.1016/j.spa.2021.07.010}
}
{
\bibitem{Heiny2019}
Heiny, J. \& Mikosch, T.:  
The eigenstructure of the sample covariance matrices of high-dimensional stochastic volatility models with heavy tails. \textit{Bernoulli} \textbf{25}, no. 4B, (2019), 3590-3622.
\url{https://doi.org/10.3150/18-BEJ1103}
}
{
\bibitem{HeinyYslas2021}
Heiny, J., Mikosch, T. \&  Yslas, J.:  
Point process convergence for the off-diagonal entries of sample covariance matrices. \textit{Ann. Appl. Probab.} \textbf{31}, no. 2, (2021), 558-560.
\url{https://doi.org/10.1214/20-AAP1597}
}
\bibitem{HuangTik2020} Huang, H. \& Tikhomirov, K.: A remark on the smallest singular value of powers of
Gaussian matrices. \textit{Electron. Commun. Probab.} \textbf{25}, no. 10, (2020), 1-8.
\url{https://doi.org/10.1214/20-ECP285}

\bibitem{Klenke} Klenke, A.: \textit{Probability theory: A comprehensive course}. Second edition, Springer-Verlag London, (2014).
\url{https://doi.org/10.1007/978-1-4471-5361-0}
{
\bibitem{Kokoszka2000}
Kokoszka, P. \& Mikosch, T.:
The periodogram at the Fourier frequencies. 
\textit{Stochastic Process. Appl.} \textbf{86}, no. 1,  (2000), 49-79.
\url{https://doi.org/10.1016/S0304-4149(99)00086-1}
}
{
\bibitem{Kostlan}
Kostlan, E.: Complexity theory of numerical linear algebra. 
\textit{J. Comput. Appl. Math.} \textbf{22}, no. 2-3, (1988), 219-230.
\url{https://doi.org/10.1016/0377-0427(88)90402-5}
}

\bibitem{Simanca} Kra, I. \& Simanca, S.: On circulant matrices. \textit{Notices Amer. Math. Soc.} \textbf{59}, no. 3, (2012), 368-377.
\url{https://doi.org/10.1090/noti804}

\bibitem{LinBai} Lin, Z. \& Bai, Z.: \textit{Probability inequalities}.
Science Press Beijing, Beijing. Springer, Heidelberg, (2010).
\url{https://doi.org/10.1007/978-3-642-05261-3}

\bibitem{Lin2009} 
Lin, Z. \& Liu, W.: 
On maxima of periodograms of stationary processes. \textit{Ann. Stat.} \textbf{37}, no. 5B,  (2009), 2676-2695.
\url{https://doi.org/10.1214/08-AOS590}

\bibitem{Meckes} Meckes, M.: Some results on random circulant matrices. \textit{Inst. Math. Stat. Collections} \textbf{5},  (2009), 213-223.
\url{https://doi.org/10.1214/09-IMSCOLL514}

\bibitem{pan2001structured} Pan, V.: \textit{Structured matrices and polynomials: unified superfast algorithms}. Birkh{\"a}user Boston, (2001).
\url{https://doi.org/10.1007/978-1-4612-0129-8}

\bibitem{panSZ2015} Pan, V., Svadlenja, J. \& Zhao, L.: Estimating the norms of random circulant and Toeplitz matrices and their inverses. \textit{Linear Algebra Appl.} \textbf{468}, (2015), 197-210.
\url{https://doi.org/10.1016/j.laa.2014.06.027}

\bibitem{Peligrad2010}  
Peligrad, M. \& Wu, W.: Central limit theorem for Fourier transforms of stationary processes. \textit{Ann. Probab.} \textbf{38}, no. 5, (2010), 2009-2022.
\url{https://doi.org/10.1214/10-AOP530}

\bibitem{Castillo} P\'erez, I., Katzav, E. \& Vivo, P.: Phase transitions in the condition-number distribution of Gaussian random matrices. \textit{Phys. Rev. E} \textbf{90},  (2014), 050103-1--050103-5.
\url{https://doi.org/10.1103/PhysRevE.90.050103}

\bibitem{RAU} Rauhut, H.: Circulant and Toeplitz matrices in compressed sensing. {\it Proc. SPARS'09-Signal Processing with Adaptive Sparse Structured Representations},  Saint-Malo, France, (2009).
\url{http://www.mathc.rwth-aachen.de/~rauhut/files/ToeplitzSPARS.pdf}

\bibitem{Rider2014} Rider, B. \&  Sinclair, C.: 
Extremal laws for the real Ginibre ensemble.
\textit{Ann.
Appl. Probab.} {\bf 24},  no. 4, (2014), 1621-1651.
\url{https://doi.org/10.1214/13-AAP958}


\bibitem{Robbins1955} Robbins, H.: 
A remark on Stirling's formula.
\textit{Amer. Math. Monthly} {\bf 62},  no. 1, (1955), 26-29.
\url{https://doi.org/10.2307/2308012}
{
\bibitem{Rudelson2010} 
Rudelson, M. \& Vershynin, R.:
Non-asymptotic theory of random matrices: extreme singular values.
\textit{Proceedings of the International Congress of Mathematicians} {\bf III}, Hindustan Book Agency, New Delhi, (2010), 1576-1602.
\url{https://doi.org/10.1142/7920}
}
{
\bibitem{Vershynin2012} 
Vershynin, R.:
\textit{Compressed sensing theory and applications.}
Cambridge University Press, (2012)
\url{https://doi.org/10.1017/CBO9780511794308.006}
}

\bibitem{Sankar} Sankar, A., Spielman, D. \& Teng, S.: Smoothed analysis of the condition numbers and growth factors of matrices. \textit{SIAM J. Matrix Anal. Appl.} \textbf{28}, no. 2,  (2006), 446-476.
\url{https://doi.org/10.1137/S0895479803436202}
{
\bibitem{SenVirag} 
Sen, A. \& Vir\'ag, B.: The top eigenvalue of the random Toeplitz matrix and the sine kernel.
\textit{Ann. Probab.} \textbf{41}, no. 6, (2013), 4050-4079.
\url{https://doi.org/10.1214/13-AOP863}
}

\bibitem{Shakil} Shakil, M. \& Ahsanullah, M.: A note on the characterizations of the
distributions of the condition numbers of real Gaussian matrices. \textit{Spec. Matrices De Gruyter} \textbf{6}, no. 1, (2018), 282-296.
\url{https://doi.org/10.1515/spma-2018-0022}

{
\bibitem{Singull2021}
Singull, M., Uwamariya, D. \&  Yang, X.:
Large-deviation asymptotics of condition numbers of random matrices. 
\textit{J. Appl. Probab.} \textbf{58}, no. 4, (2021), 1114-1130. 
\url{https://doi.org/10.1017/jpr.2021.13}
}
\bibitem{Smale} Smale, S.: On the efficiency of algorithms of analysis. \textit{Bull. Amer. Math Soc.} \textbf{13}, no. 2,  (1985), 87-121.
\url{https://doi.org/10.1090/S0273-0979-1985-15391-1}
{
\bibitem{Soshnikov2004}
Soshnikov, A.: Poisson statistics for the largest eigenvalues of Wigner random
matrices with heavy tails. 
\textit{Electron. Comm. Probab.} \textbf{9}, (2004), 82-91.
\url{https://doi.org/10.1214/ECP.v9-1112}
}
{
\bibitem{Soshnikov2006}
Soshnikov, A.: \textit{Poisson statistics for the largest eigenvalues in random matrix
ensembles}. 
Mathematical physics of quantum mechanics 351-364,
Lecture Notes in Phys. \textbf{690}, Springer, Berlin, (2006).
\url{https://doi.org/10.1007/3-540-34273-7_26}
}
{
\bibitem{Szarek}
Szarek, S.: 
Condition numbers of random matrices.
\textit{J. Complexity} \textbf{7}, no. 2, (1991), 131-149.
\url{https://doi.org/10.1016/0885-064X(91)90002-F}
}

{\bibitem{TaoVu2010}
Tao, T. \& Vu, V.: Random matrices: the distribution of the smallest singular values. 
\textit{Geom. Funct. Anal.} \textbf{20}, no. 1, (2010), 260-297.
\url{https://doi.org/10.1007/s00039-010-0057-8}
}
{
\bibitem{Tatarko}
Tatarko, K.: 
An upper bound on the smallest singular value of a square random matrix. 
\textit{J. Complexity} \textbf{48}, (2018), 119-128.
\url{https://doi.org/10.1016/j.jco.2018.06.002}
}
\bibitem{Turing} Turing, A.: Rounding-off errors in matrix processes. \textit{Quart. J. Mech. Appl. Math.} \textbf{1}, no. 1,  (1948), 287-308.
\url{https://doi.org/10.1093/qjmam/1.1.287}

\bibitem{Turkman1984} 
Turkman, K. \& Walker, M.: On the Asymptotic distributions of maxima of trigonometric polynomials with
random coefficients. \textit{Adv. Appl. Probab.} \textbf{16}, no. 4,  (1984), 819-842.
\url{https://doi.org/10.2307/1427342}

\bibitem{VisTre} Viswanath, D. \& Trefethen, L.: Condition numbers of random triangular
matrices. \textit{SIAM J. Matrix Anal. Appl.} \textbf{19}, no. 2, (1998), 564-581.
\url{https://doi.org/10.1137/S0895479896312869}

\bibitem{vonNG1}  von Neumann, J. \& Goldstine, H.: Numerical inverting of matrices of high order. \textit{Bull. Amer. Math. Soc.} \textbf{53}, no. 11, (1947), 1021-1099.
\url{https://doi.org/10.1090/S0002-9904-1947-08909-6}

\bibitem{Wozniakowski} 
Wo\'zniakowski, H.: Numerical stability for solving nonlinear equations. \textit{Numer. Math.} \textbf{27},  (1977), 373-390.
\url{https://doi.org/10.1007/BF01399601}
{
\bibitem{Yakir2020} 
Yakir, O. \& Zeitouni, O.: The minimum modulus of Gaussian trigonometric polynomials. 
\textit{Israel J. Math.}, (2021), 1-24.
\url{https://doi.org/10.1007/s11856-021-2218-x}
}


\end{thebibliography}

\end{document}